\documentclass[12pt,notitlepage,a4paper]{article}
\pdfoutput=1

\usepackage[left=2.5cm,right=2.5cm,top=2.5cm,bottom=2.5cm]{geometry}
\usepackage{graphicx}
\usepackage{amssymb,mathtools, amsmath, amsfonts, amsthm}

\usepackage{float}
\usepackage{hyperref}
\usepackage{enumitem}
\usepackage{chngcntr}
\usepackage[noabbrev]{cleveref}

\makeatletter
\newtheorem*{rep@theorem}{\rep@title}
\newcommand{\newreptheorem}[2]{%
	\newenvironment{rep#1}[1]{%
		\def\rep@title{#2 \ref{##1}}%
		\begin{rep@theorem}}%
		{\end{rep@theorem}}}
\makeatother

\counterwithout{equation}{section}

\newlength{\margen}
\hypersetup{
	colorlinks=false,
	pdfborder={1 1 0.0005},
}
\usepackage{tikz-cd}
\usetikzlibrary{cd}
\usepackage[english]{babel}
\usepackage{todonotes}
\usepackage{cleveref}
\usepackage{caption}
\usepackage{subcaption}
\usepackage{bbding}
\usepackage{tcolorbox}
\usepackage[square,numbers]{natbib}

\theoremstyle{definition}

\newtheorem{theorem}{Theorem}[section]
\newtheorem{obs}{Observation}[section]
\newtheorem{lemma}{Lemma}[section]
\newtheorem{corollary}{Corollary}[section]
\newtheorem{definition}{Definition}[section]

\newtheorem*{theorem*}{Theorem}
\newreptheorem{theorem}{Theorem}
\newreptheorem{corollary}{Corollary}

\newcommand{\R}{\mathbb{R}}
\newcommand{\N}{\mathbb{N}}

\newcommand{\Ln}{\lim\limits_{n\to \infty}}

\newcommand{\morph}[1]{\sim_#1}

\newcommand{\ehr}{\textsc{Ehr}}
\newcommand{\PR}[1]{\mathrm{Pr}\left(#1\right)}
\newcommand{\Tr}{\mathrm{Tr}}
\newcommand{\InR}[1]{\left\{ #1_R \right\}_{R\in \sigma}}

\newcommand{\aut}{\mathrm{aut}}
\newcommand{\len}{\mathrm{len}}
\newcommand{\qr}{\mathrm{qr}}
\newcommand{\ex}{\mathrm{ex}}
\DeclarePairedDelimiter\floor{\lfloor}{\rfloor}

\begin{document}
\thispagestyle{plain}
\begin{center}
	\large
	\textbf{Probabilities of first order sentences on sparse random relational structures: An application to definability on random CNF formulas}
\end{center}
	
\vspace{0.4cm}
\begin{flushright}
	\normalsize
	L\'azaro Alberto Larrauri \\
	\textit{Universitat Polit\`ecnica de Catalunya}\\
	\small
	lazaro.alberto.larrauri@upc.edu\\	
	\normalsize
\end{flushright}	

\vspace{0.9cm}

\large
\noindent
\textbf{Abstract}
\normalsize
\vspace{0.4cm}\\
We extend the convergence law for sparse random graphs proven
by Lynch to arbitrary relational languages.
We consider a finite relational vocabulary $\sigma$
and a first order theory $T$ for $\sigma$ 
composed of symmetry and 
anti-reflexivity axioms. We define a binomial random model of finite 
$\sigma$-structures that satisfy $T$ and show that first order properties have 
well defined asymptotic probabilities when the expected number of tuples satisfying
each relation in $\sigma$ is linear.
It is also shown that these limit probabilities are well-behaved with
respect to several parameters that represent the density of tuples in each relation $R$ 
in the vocabulary $\sigma$. 
An application of these results to the problem of random Boolean 
satisfiability is presented. 
We show that in a random $k$-CNF formula on $n$ variables, where
each possible clause occurs with probability $\sim c/n^{k-1}$, independently any 
first order property of $k$-CNF formulas that
implies unsatisfiability does almost surely not hold as $n$ tends to infinity.
   
\noindent
\vspace{0.5 cm} 

\clearpage

\section*{Introduction}

We say that a sequence of random structures $\{G_n\}_n$ satisfies a limit
law with respect to some logical language $L$ if for every property $P$ expressible in
$L$ the probability that $G_n$ satisfies $P$ tends to some limit as $n \to \infty$.
If that limit takes only the values zero and one then we say that $\{G_n\}_n$ satisfies
a zero-one law with respect to $L$.  \par
Convergence and zero-one laws have been extensively studied on the binomial graph $G(n,p)$.
The seminal theorem on this topic, due to Fagin \cite{fagin1976probabilities}
and Glebskii et al. \cite{glebskii1969range} independently, 
concerns general relational structures.
When applied to graphs it states that if $p$ is fixed,
then $G(n,p)$ satisfies a zero-one law with respect to the first order
(FO) language of graphs. \par
This zero-one law was later extended by Shelah and Spencer
in \cite{shelah1988zero}. There it is proven, among other results,
that if $p:=p(n)$ is a decreasing function
of the form $n^{-\alpha}$ and $\alpha>0$ is irrational,
then $G(n,p(n))$ obeys a zero-one law with respect to FO logic. Moreover, 
it is also proven that if $\alpha\in (0,1)$ is rational then $G(n,p(n))$ 
does not obey a convergence law. \par
This was further studied by Lynch
in \cite{lynch1992probabilities}, where it is shown that in the
case where the expected number of edges is linear, i.e. when $p(n)\sim \beta/n$
for some $\beta >0$, then $G(n,p(n))$ satisfies a limit law with
respect to FO logic. The following is a restatement of the main result
in that article.

\begin{theorem*}[Lynch, 1992]
	Let $p(n)\sim \beta/n$. For every FO sentence $\phi$, the function
	$F_\phi: (0,\infty)\rightarrow [0,1]$ given by 
	\[ F_\phi(\beta) = \Ln \mathrm{Pr}\left( \mathrm{G}(n,p(n))
	\text{ satisfies } \phi   \right) \]
	is well defined and is given by an expression with parameter $\beta$ built
	using rational constants, addition, multiplication and exponentiation with base $e$.
\end{theorem*}

A relevant aspect of this result is that the limit probability of any FO
property in $G(n,p(n))$ when $p(n)\sim \beta/n$ varies analytically with $\beta$.
A consequence of this is that FO logic cannot ``capture" sudden changes 
in the structure of $G(n,p(n))$. \par
It was left open at the end of \cite{lynch1992probabilities} whether the convergence law obeyed
by $G(n,p(n))$ in the range $p(n)\sim \beta/n$
could be generalized to other random models of relational structures 
that contain relations of arity greater than $2$.
A result in this direction was obtained in \cite{salvadorbrasil}, 
among other zero-one and convergence laws.
They consider the random model of $d$-uniform hypergraphs $G^d(n,p)$ where
each $d$-edge is added to a set of $n$ labeled vertices independently with probability $p$.
It is shown that when $p(n)\sim \beta/n^{d-1}$, i.e. when the expected number of
edges is linear, $G^d(n,p(n))$ obeys a convergence law with respect to the FO
language of $d$-uniform hypergraphs. With little additional work it can be shown that
in these conditions the limit probability of any FO property of $G^d(n,p(n))$ varies 
analytically with $\beta$. We 
extend this result to arbitrary relational structures
on whose relations we can impose symmetry and anti-reflexivity constraints 
(\Cref{thm:main}). \par

This generalization is motivated by an application to the problem of random 
SAT. We continue the study started by Atserias in 
\cite{atserias2005definability}
with respect to the definability in first order
logic of certificates for unsatisfiability
that hold for typical unsatisfiable formulas. 
A random model for $3$-CNF formulas where 
each possible clause over $n$ variables 
is added independently with probability $p$ is considered there.
In this model the expected number of clauses $m$ is $\Theta(n^3p)$ as $n$ grows.
The main result of that article states the following:
(1) if $m=\Theta(n^{2-\alpha})$ 
for an irrational number $\alpha>0$, then no FO property of $3$-CNF formulas
that implies unsatisfiability holds asymptotically almost surely (a.a.s.) 
for unsatisfiable formulas, and (2) if $m=\Theta(n^{2+\alpha})$ for $\alpha>0$,
then there exists some FO property that implies
unsatisfiability and  holds a.a.s. for unsatisfiable formulas. \par
The second part of the statement is the simpler one to prove:
it can be shown that when $m=\Theta(n^{2+\alpha})$ for some 
$\alpha>0$ the random $3$-CNF formula
a.a.s. contains some fixed unsatisfiable subformula (which depends
on the choice of $\alpha$). This is clearly expressible in FO logic, 
so (2) follows. The proof of (1) is more involved and, in fact, 
shows something stronger: if $m=\Theta(n^{2-\alpha})$
for $\alpha>0$ irrational, then all FO properties 
that imply unsatisfiability
a.a.s. do not hold. This proof employs
techniques based in those used by Shelah and Spencer 
in \cite{shelah1988zero} to prove that $G(n,p)$ satisfies a 
zero-one law with respect to FO logic when $p$ is an irrational 
power of $n$. \par
Since the techniques used to prove (2) rely on the fact that $\alpha$
is irrational, the study of the range $m=\Theta(n)$
(that is, $m=\Theta(n^{2-\alpha})$ with $\alpha=1$), was left open.
This range is of special interest because it is where the phase 
transition from almost sure satisfiability to almost sure unsatisfiability
takes place. It was shown
in \cite{chvatal1992mick} that a random $k$-CNF formula with
$m$ clauses over $n$ variables satisfying that $m\sim c  n$ 
is a.a.s satisfiable for all sufficiently small values of $c$ and is a.a.s
unsatisfiable for all sufficiently large values of $c$.
\par
The possibility of studying FO definability of certificates for unsatisfiability
in random $l$-CNF formulas with a linear expected number of clauses 
using a generalization of Lynch  theorem was suggested by Atserias. 
This application is discussed in \Cref{sec:SAT}. We give a brief overview
of it here. Let $F(l,n,p)$ be a random model of $l$-CNF 
formulas where each $l$-clause over $n$ variables is chosen independently 
with probability
$p$. Let $F^l_n(\beta)$ denote a random formula in
$F(l,n,p)$ where $p:=p(n)\sim \beta/n^{l-1}$. 
Suppose that 
every FO property of $l$-CNF formulas 
has a well defined asymptotic probability in $F^l_n(\beta)$
for any $\beta>0$.
Further suppose that these asymptotic probabilities vary analytically
with $\beta$. Then any FO property that implies unsatisfiability
a.a.s does not hold in $F^l_n(\beta)$ for
$\beta>0$. Indeed, let $P$ be one such FO property. One can find a
value $\beta_0>0$ satisfying that 
a.a.s $F^l_n(\beta)$ is satisfiable when $0<\beta<\beta_0$.
As a consequence $P$ a.a.s does not hold
in $F^l_n(\beta)$ when $0<\beta<\beta_0$.
Since the asymptotic probability of $P$ varies analytically
with $\beta$ and it vanishes in the non-empty interval $(0,\beta_0)$,
because of the Principle of analytical continuation
it must be true that a.a.s $P$ does not hold in $F^l_n(\beta)$ for all
$\beta>0$.

\section{Preliminaries}

\subsection{General notation}\label{subsect:notation}

Given a positive natural number $n$, we write
$[n]$ to denote the set ${1,2,\dots,n}$.
Given numbers, $n,m\in \N$ with $m\leq n$ we denote by
$(n)_m$ the $m$-th falling factorial of $n$. 
Given a set $S$ and a natural number $k\in \N$
we use $\binom{S}{k}$ to denote the set of 
subsets of $S$ of size $k$.Given a set
$S$ and $n\leq |S|$, we define
$(S)_n$ as the subset of $S^n$ consisting of the $n$-tuples
whose coordinates are all different. 
We also define $S^*:=\bigcup_{n=0}^\infty S^n$ and
$(S)_*:=\bigcup_{n\leq |S|} (S)_n$. \par

We use the convention that over-lined variables, 
like $\overline{x}$, denote ordered tuples of arbitrary length.
Given an ordered tuple $\overline{x}$
we define $\len(\overline{x})$ as its length. 
Given a tuple $\overline{x}$ and an element $x$ the expression
$x\in \overline{x}$ means that $x$ appears as some coordinate
in $\overline{x}$. 
Given a map $f:X\rightarrow Y$ and 
an ordered tuple $\overline{x}:=(x_1,\dots,x_a)\in X^*$ 
we define $f(\overline{x})\in Y^*$ as the tuple 
$(f(x_1),\dots,f(x_a))$.
Given two tuples $\overline{x},\overline{y}$
we write $\overline{x}^\smallfrown \overline{y}$ to denote their 
concatenation. Given a set $S$ and elements $x_s$ for each $s\in S$
we write $\{x_s\}_{s\in S}$, or just $\{x_s\}_s$ when $S$ is understood,
to denote the tuple indexed by $S$ which
contains the element $x_s$ at the position given by $s$. 
\par
Let $S$ be a set, $a$ a positive natural number, 
and $\Phi$ a group of permutations over 
$[a]$. Then $\Phi$ acts naturally on
$S^a$ in the following way: Given $g\in \Phi$ and
$\overline{x}:=(x_1,\dots,x_a)\in S^a$ let 
$g  \overline{x}:=(x_{g(1)},\dots,x_{g(a)})$. 
We denote by $S^a/\Phi$ the quotient
of $S^a$ by this action. Given
$\overline{x}:=(x_1,\dots, x_a)\in S^a$ we denote its equivalence
class in $S^a/\Phi$ by $[x_1,\dots,x_a]$ or $[\overline{x}]$.
Thus, for
$g\in \Phi$, by definition $[x_1,\dots,x_a]=[x_{g(1)}
,\dots,x_{g(a)}]$. \par
The notations $\overline{x}$ and
$(x_1,\dots, x_a)$ 
represent
ordered tuples while 
$[\overline{x}]$ and
$[x_1,\dots,x_a]$ denote ordered tuples modulo the
action of some arbitrary group of permutations. Which group it is
will depend on the ambient set where $[x_1,\dots,x_a]$ belongs
and it should either be clear from context or not be relevant.\par

Given real functions over the natural numbers 
$f,g:\N \rightarrow \R$ the expressions $f=O(g)$,
$f=o(g)$ and $f=\Theta(g)$ have their usual meaning.
If $g(n)\neq 0$ for $n$ large enough 
we write $f\sim g$ if $\Ln \frac{f(n)}{g(n)}=1$. \par

%
\subsection{Probabilistic preliminaries}

We assume familiarity with basic probability theory. 
We denote by $\mathrm{Poiss}_\lambda(n)$ the discrete probability mass
function of a random Poisson variable with
mean $\lambda$. That is, $\mathrm{Poiss}_\lambda(n)=e^{-\lambda}\frac{\lambda^n}{n!}$.
We define $\mathrm{Poiss}_\lambda(\geq n)=1 - \sum_{i=0}^{n-1} \mathrm{Poiss}_\lambda(i)$. \par
Given some sequence of events $\{A_n\}_n$ we say that $A_n$ is 
satisfied asymptotically
almost surely (a.a.s.) if $\Pr(A_n)$ tends to $1$ as $n\to \infty$.
Given a sequence of random variables $\{X_n\}_n$, the \textbf{first moment method}
is an application of Markov's inequality that establishes that if $\mathrm{E}[X_n]$
tends to zero as $n\to \infty$ then a.a.s $X_n=0$.\par
If $A,B$ are events we may write the conditioned probability $\mathrm{Pr}(A\, | \, B )$ 
as $\mathrm{Pr}_B(A)$ to shorten some expressions. In this situation, 
given a random variable $X$ we put $\mathrm{E}_B[X]$ to denote conditional expectation of $X$ given the event $B$. \par

Our main tool for proving the convergence in distribution 
to Poisson variables is the next result, which can
be found in \cite[Theorem 1.23]{bollobas2001random}.

\begin{theorem} \label{thm:BrunSieve}
	Let $l\in \N$. For each 
	$n\in \N$, let $X_{n,1},\dots, X_{n,l}$ be non-negative
	random integer variables over the same
	probability space. Let $\lambda_1,\dots,\lambda_l$ 
	be real numbers. Suppose for any $r_1,\dots,r_l \in \N$
	\[ 
	\Ln \mathrm{E}\left[
	\prod_{i=1}^{l} \binom{X_{n,i}}{r_i} \right]
	= \prod_{i=1}^{l} \frac{\lambda_i}{r_i !}.	
	\]
	Then the $X_{n,1},\dots,X_{n,l}$ converge in distribution to
	independent Poisson variables with means $\lambda_1,\dots,\lambda_l$ 
	respectively. 
\end{theorem}

We use the following observation in order to compute the binomial moments of 
our random variables.

\begin{obs} \label{obs:binomialmean} Let $X_1,\dots, X_l$ be non negative
	random integer variables over the same probability space. 
	Let $r_1,\dots,r_l\in \N$.	Suppose
	each $X_i$ is the sum of indicator random variables
	(i.e. variables that only take the values $0$ and $1$)
	$X_i=\sum_{j=1}^{a_i} Y_{i,j}$. Define 
	$\Omega:=\prod_{i=1}^l \binom{[a_i]}{r_i}$. That is,
	the elements $\{S_i\}_{i\in[l]}\in \Omega$
	represent all the possible unordered choices of 
	$r_i$ indicator variables $Y_{i,j}$ for each $i\in [l]$.
	Then 
	\[
	\mathrm{E}\left[
	\prod_{i=1}^{l} \binom{X_i}{r_i}\right]=
	\sum_{\{S_i\}_{i\in [l]}} \mathrm{Pr}\left(
	\bigwedge_{\substack{i\in [l]\ j\in S_i}} Y_{i,j}=1
	\right).	
	\]
\end{obs}

\subsection{Logical preliminaries}
We assume familiarity with first order logic (FO). We follow 
the convention that first order logic contains the equality symbol. 
Given a vocabulary $\sigma$ we denote by $FO[\sigma]$ the set of 
first order formulas of vocabulary $\sigma$.
Given a relation symbol $R\in \sigma$ we denote by $ar(R)$ the arity of $R$. 
Given a formula $\phi\in FO[\sigma]$ we use the notation $\phi(\overline{y})$ 
to express that $\overline{y}$ is a tuple of 
(different) variables which contains all free variables in $\phi$ and
none of its bounded variables, but it may contain variables
which do not appear in $\phi$.
Formulas with no free variables are called \textbf{sentences} and 
formulas with no quantifiers are called \textbf{open formulas}. The \textbf{quantifier 
rank} of a formula $\phi$, written as $\qr(\phi)$, is the maximum number of 
nested quantifiers in $\phi$.
We call \textbf{edge sentence} to any consistent open formula that contains
no occurrence of the equality symbol `$=$'.

\subsection{Structures as multi-hypergraphs} \label{sect:structures}

For the rest of the article consider fixed:
\begin{itemize}
	\item A relational vocabulary $\sigma$ such 
	that all the relations $R\in\sigma$ satisfy $ar(R)\geq 2$. 
	\item 
	Groups $\{ \Phi_R \}_{R\in \sigma}$
	such that each $\Phi_R$ is consists of 
	permutations on $[ar(R)]$ with the usual 
	composition as its operation.	
	\item 
	Sets $\{P_R\}_{R\in \sigma}$ satisfying
	$P_R\subseteq \binom{[ar(R)]}{2}$ for 
	all $R\in \sigma$.	
	
\end{itemize}
%


We define $\mathcal{C}$ as the 
class of $\sigma$-structures that
satisfy the 
following axioms: 
\begin{itemize}
	\item \textit{Symmetry axioms}: For each $R\in \sigma$ and
	$g\in \Phi_R$:
	\[ \forall \overline{x}:=x_1,\dots, x_{ar(R)} \left(  R(\overline{x})
	\iff R(g \overline{x}) \right)    \]
	\item \textit{Anti-reflexivity axioms}: For each 
	$R\in \sigma$ and $\{i,j\}\in P_R$
	\[ \forall x_1,\dots, x_{ar(R)} 
	\left( (x_i=x_j) \implies \neg R(x_1,\dots, x_{a_s})
	\right)\]
\end{itemize}

Structures in $\mathcal{C}$ generalize the usual notion of a hypergraph
in the sense that they contain multiple ``adjacency'' relations with arbitrary 
symmetry and anti-reflexivity axioms. \par
We use the usual graph theory nomenclature and notation with some minor changes. 
In the scope of this article \textbf{hypergraphs}
are structures in $\mathcal{C}$.
Given a hypergraph $G$ its \textbf{vertex set} $V(G)$
is its universe.\par
In order to define the edge sets of $G$ we need the following auxiliary definition
\begin{definition} 
	Let $V$ be a set, and let $R\in \sigma$.
	We define the \textbf{set of possible edges over $V$
	given by $R$} as
	\[ E_R[V]= (V^{ar(R)}/\Phi_R)\, \setminus \, X, \]
	where
	\[
	X=
	\Big\{ [v_1,\dots,v_{ar(R)}]  
	\quad \Big| \quad
	v_1,\dots,v_{ar(R)}\in V, \,
	\text{ and } 
	 \, v_i=v_j \text{ for some } 
	\{i,j\}\in P_R \Big\}.
	\]
	We call \textbf{edges} to the elements of
	$E_R[V]$ and we say that the \textbf{sort} of an edge $e\in E_R[V]$
	is $R$.	In the case where $V=[n]$ we write simply $E_R[n]$ instead
	of $E_R[[n]]$
\end{definition}

That is, $E_R[V]$ contains all the ``$ar(R)$-tuples of elements
in $V$ modulo the permutations
in $\phi_R$" 
excluding those that contain some repetition of elements in
the positions given by $P_R$.\par
Let $G$ be a hypergraph with vertex set is $V$ and let $R\in \sigma$
be a relation. 
We define the \textbf{edge set of $G$ given by $R$}, denoted by $E_R(G)$, 
as the set of edges $[\overline{v}]\in E_R[V]$ such that $\overline{v}\in R^G$. 
We define \textbf{the total edge set of $G$} as the set $E(G):=\cup_{R\in \sigma} E_R(G)$.
Given an edge, $e\in E(G)$ we denote by $V(e)$
the set of all vertices that participate in $e$. 
\par
Clearly a hypergraph $G$ is completely given by its vertex set $V(G)$ and its edge 
set $ E(G)$. Notice that edges $e\in E(G)$ are sorted according to the relation they represent. The \textbf{size} of $G$, written as $|G|$, is its number of vertices. 
\par

Given two hypergraphs $H$ and $G$ we say that $H$
is a \textbf{sub-hypergraph} of $G$, written as $H\subset G$,
if $V(H)\subset V(G)$ and $E(H)\subset E(G)$ (notice 
that this is equivalent to $E_R(H)\subset E_R(G)$ for all
$R\in \sigma$, since the edges are sorted).\par

Given a set of vertices $U\subseteq V(G)$, 
we denote by $G[U]$ the \textbf{hypergraph induced
by $G$ on $U$}. That is, $G[U]$ is a hypergraph
$H=(V(H),\{E(H)_R\}_{R\in \sigma})$ such that 
$V(H)=U$ and for any $R\in \sigma$ 
an edge $e\in E_R(G)$ belongs 
to $E_R(H)$ if and only if $V(e)\subset U$.
\par

We define the \textbf{excess} $\ex(G)$ of a hypergraph $G$ as the number
\[
\ex(G):= \left(\sum_{R\in \sigma} (ar(R)-1)|E_R(G)|\right) - |V(G)|.  
\] 
That is, the excess of $G$ is the "weighted number of edges"
minus its number of vertices. \par
An hypergraph $G$ is \textbf{connected} if for any two vertices $v,u\in V(G)$
there is a sequence of edges $e_1,\dots, e_m\in E(G)$ such that
$v\in V(e_1), u\in V(e_m)$ and for each $i\in [m-1]$, 
$V(e_i)\cap V(e_{i+1})\neq \emptyset$. It holds that
$\ex(G)\geq -1$ for any connected hypergraph.
\par
Given a hypergraph $G$ we define the following metric, $d$, over $V(G)$:
\[ d^G(u,v)= \min_{\substack{H \subset G\ 
		H \text{ connected }\\
		u,v\in V(H)}} |E(H)| .\]
That is, the \textbf{distance} between $v$ and $u$ is the minimum number of
edges necessary to connect $v$ and $u$. 
If such number does not exist we define $d^G(u,v)=\infty$. 
When $G$ 
is understood or not relevant we simply write $d$ instead of $d^G$.
Equivalently, the distance $d$
coincides with the usual one defined over the Gaifman graph of the structure 
$G$. The \textbf{diameter} of a hypergraph is the maximum distance between any 
pair of vertices. 
We extend naturally the distance $d$ to sets and tuples of
vertices, as usual. Given a vertex/set/tuple $X$ and a number
$r\in \N$ we define the \textbf{neighborhood}
$N^G(X;r)$, or simply $N(X;r)$ when $G$ is not relevant,
as the set of vertices $v$ such that $d^G(X,v)\leq r$.
\par
A connected hypergraph $G$ is a \textbf{path} between two of its 
vertices $v,u\in V(G)$ if $G$ 
does not contain any connected proper sub-hypergraph containing both $v,u$.
A connected hypergraph $G$ is a \textbf{tree} if $\ex(G)=-1$ and \textbf{dense} if $\ex(G)>0$.
An hypergraph is called $r$-\textbf{sparse} if it does not contain any dense sub-hypergraph $H$
such that $diam(H)\leq r$.
A connected hypergraph $G$ with $\ex(G)\geq 0$ is called \textbf{saturated} 
if for any non-empty proper sub-hypergraph
$H\subset G$ it holds $\ex(H)<\ex(G)$. 
A connected hypergraph $G$ with $\ex(G)=0$ is called a \textbf{unicycle}. 
A saturated unicycle is called a \textbf{cycle}. We say that an edge $e:=[\overline{v}]$
contains a \textbf{loop} if some vertex $v$ appears in $\overline{v}$ more than once.\par
A \textbf{rooted tree} $(T,v)$ is a tree $T$ with a 
distinguished vertex $v\in V(T)$ called its \textbf{root}.
We usually omit the root when it is not relevant and 
write just $T$ instead of $(T,v)$. The
\textbf{initial edges} of a rooted tree $(T,v)$ are 
the edges in $T$ that contain $v$. 
We define the radius of a rooted tree
as the maximum distance between its root
and any other vertex.
\par

Let $\Sigma$ be a set. A \textbf{$\Sigma$-hypergraph}
is a pair $(H, \chi)$ where $H$ is a hypergraph
and $\chi: V(H)\rightarrow \Sigma$ is a map 
called a \textbf{$\Sigma$-coloring} of $H$. \par

\textbf{Isomorphisms} between hypergraphs are defined as 
isomorphisms between relational structures. Isomorphisms between $\Sigma$-hypergraphs are just
isomorphisms between the underlying hypergraphs that also preserve their colorings. 
In both cases we denote the isomorphism relation by $\simeq$. 
Given a hypergraph $H$, resp. a $\Sigma$-hypergraph $(H, \chi)$,
an \textbf{automorphism} of $H$, resp. $(H,\chi)$,
is an isomorphism from $H$, resp. $(H,\chi)$, to itself.
We  denote by $\aut(H)$, resp. $\aut(H,\chi)$,
the number of such automorphisms. \par

Let $H$ be a hypergraph and let $V$ be a set. We define the
set of \textbf{copies of $H$ over $V$}, denoted as $Copies(H,V)$, 
as the set of hypergraphs 
$H^\prime$ such that
$V(H^\prime)\subset V$ and $H\simeq H^\prime$. Let 
$\chi$ be a $\Sigma$-coloring of $H$. 
Analogously, we define the set $Copies\left(
(H,\chi),\, \, V\right)$ as the set of $\Sigma$-hypergraphs
$(H^\prime,\chi^\prime)$ satisfying $V(H^\prime)\subset V$ and
$(H,\chi)\simeq (H^\prime,\chi^\prime)$. Let $\mathbb{H}$ be 
an isomorphism class of $\Sigma$-hypergraphs. Then the set
$Copies(\mathbb{H}, V)$ is defined as the set of $\Sigma$-hypergraphs
$(H^\prime,\chi^\prime)$ such that 
$V(H^\prime)\subset V$ and
$(H^\prime,\chi^\prime)\in \mathbb{H}$. 
Let $v\in V$ and $s\in \Sigma$. We define the
set $Copies\left(\mathbb{H}, V;\,\, (v,s)\right)$ 
as the set of $\Sigma$-hypergraphs
$(H^\prime,\chi^\prime)\in Copies(\mathbb{H}, V)$
that satisfy $v\in V(H^\prime)$ as well as
$\chi^\prime(v)=s$. \par

Given $\mathbb{H}$ an isomorphism class of hypergraphs or $\Sigma$-hypergraphs,
we define expressions
such as $\ex(\mathbb{H})$, $\aut(\mathbb{H})$,
$|V(\mathbb{H})|$, $|E(\mathbb{H})|$ or
$Copies(\mathbb{H},V)$ via representatives of $\mathbb{H}$.\par

\subsection{Ehrenfeucht-Fraisse Games}

We assume familiarity with Ehrenfeucht-Fraisse (EF) games.
An introduction to the subject can be found for instance in
\cite[Section 2]{finitemodeltheory1}, for example. Given
hypergraphs $H_1$ and $H_2$ we denote the $k$-round EF game played 
on $H_1$ and
$H_2$ by $\ehr_k(H_1;H_2)$.
The following is satisfied:

\begin{theorem}
	[Ehrenfeut, \citealp{ehrenfeucht1961application}] Let
	$H_1$ and $H_2$ be hypergraphs.
	Then Duplicator wins $\ehr_k(H_1;H_2)$
	if and only if $H_1$ and $H_2$ satisfy the same 
	sentences $\phi\in FO[\sigma]$ with $\qr(\phi)\leq k$.		
\end{theorem}

Given lists $\overline{v}\in V(H_1)^*$, and 
$\overline{u}\in V(H_2)^*$ of the same length, 
we denote the $k$ round 
Ehrenfeucht-Fraisse game on $H_1$ and $H_2$ with initial position given
by $\overline{v}$ and $\overline{u}$ by $\ehr_k(H_1,\overline{v};H_2,\overline{u})$.\par

We also define the $k$-round distance Ehrenfeucht-Fraisse game on 
$H_1$ and $H_2$, denoted by $d\ehr_k(H_1;H_2)$, the same way as
$\ehr_k(H_1;H_2)$, but now in order for Duplicator to win the
game the following additional condition has to be satisfied 
at the end: For any $i,j\in [k]$, $d^{H_1}(v_i,v_j)=d^{H_2}(u_i,u_j)$,
where $v_s$ and $u_s$ denote the vertex played on $H_1$, resp. $H_2$ in the 
$s$-th round of the game. 
Given $\overline{v}\in V(H_1)^*$, and $\overline{u}\in V(H_2)^*$
lists of vertices of the same length,
we define the game 
$d\ehr_k(H_1,\overline{v};H_2,\overline{u})$ analogously to 
$\ehr_k(H_1,\overline{v};H_2,\overline{u})$.

\subsection{The random model} \label{sect:random}

For each $R\in \sigma$ let
$p_R$ be a real number between zero and one.
The random model $G^{\mathcal{C}}\left(n,\InR{p}\right)$ 
is the discrete probability space that
assigns to each hypergraph $G$ whose vertex
set $V(G)$ is $[n]$ the following probability:

\[ \mathrm{Pr}(G)=\prod_{R\in \sigma} p_R^{|E_R(G)|}
(1-p_R)^{ \big|E_R[n]\big|-\big|E_R(G)\big|}.	
\]
Equivalently, this is the probability space obtained by 
assigning to each edge $e\in E_R[n]$ probability 
$p_R$ independently for each $R\in \sigma$. \par

As in the case of Lynch  theorem, we are interested in the
"sparse regime" of $G^\mathcal{C}(n,\{p\}_R)$, were the 
expected number of edges of each sort is linear. 
This is achieved when for each $R\in \sigma$ 
it holds $p_R(n)\sim \beta_R/n^{ar(R)-1}$ for some 
$\beta_R>0$.
We write $G_n\left(\{\beta_R\}_R\right)$
to denote a random sample of 
$G^\mathcal{C}\left(n,\{p_R\}_{R}\right)$
when the probabilities $p_R$ satisfy
$p_R(n)\sim \beta_R/n^{ar(R)-1}$.
When the choice of $\{\beta\}_R$ is not relevant
we write $G_n$ instead of 
$G_n\left(\{\beta_R\}_R\right)$.\par

\subsection{Main definitions}

Our main definition follow closely the ones in \cite{lynch1992probabilities}
adapted to the context of hypergraphs. 

\begin{definition} 
Let $H$ be a connected hypergraph. Then
$H$ contains a unique maximal saturated sub-hypergraph $H^\prime$
satisfying satisfies $\ex(H^\prime)=\ex(H)$ if $\ex(H)\geq 0$, and $H^\prime=\emptyset$ 
otherwise. Given
$\overline{v}\in V(H)^*$ we define $Center(H,\overline{v})$ as the 
minimal connected sub-hypergraph in $H$ that contains both $H^\prime$
and the vertices in $\overline{v}$. 
If $H$ is not connected 
we define $Center(H,\overline{v})$,
as the union of $Center(H^{\prime\prime},\overline{u})$
for all connected components $H^{\prime\prime} \subset H$,
where $\overline{u}\in V(H)^*$ contains exactly the vertices
in $\overline{v}$ belonging to $V(H^{\prime\prime})$. When $\overline{v}$
is empty we simply write $Center(H)$.
\end{definition}

\begin{definition}
Let $H$ be a hypergraph, $\overline{v}\in V(H)^*$
and $r\in \N$. Let $X$ be the set of vertices $v\in V(H)$
that either belong to $\overline{v}$ or belong 
to some saturated sub-hypergraph of $H$ with
diameter at most $2r+1$. 
We define $Core(H, \overline{v};r)$ as $N(X;r)$. If
$\overline{v}$ is empty we write $Core(H;r)$.
We say that $H$ is \textbf{$r$-simple} if all connected components of 
$Core(H;r)$ are unicycles. 	
\end{definition}

\begin{definition}\label{def:TrOperator}
	Let $H$ be a hypergraph, let $\overline{v}\in V(H)^*$ and 
	let $v\in H$ be such that $d(Center(H,\overline{v}),$ $ \, v)<\infty$. 
	Let $X\subset V(H)$ be the set
	\[
	X:=\big\{\, u\in V(H) \,\, \big| \, 
	d\left(\, Center(H, \overline{v})
	, \, \, u \right)= d\left(\, Center(H, \overline{v})
	, \, \, v \right) + d(v,u)
	\big\}.
	\]
	Then we define $\mathrm{Tr}\left(
	H,\overline{v};\,\, v\right)$ as the tree $H[X]$ 		
	with $v$ as a root. That is,
	$\mathrm{Tr}\left(
	H,\overline{v};\,\, v\right)$ is the tree formed of all 
	vertices whose only path to $Center(H,\overline{v})$ 
	contains $v$. One can easily check that $H[X]$ is indeed a tree:
	if it were not then it would contain some saturated sub-hypergraph, leading
	to a contradiction. 
	Given $r\in \N$ we define $\mathrm{Tr}(H,\overline{v};\, v;\, r)$ as
	$\mathrm{Tr}( Core(H,\overline{v};\,r),\overline{v}; \, v)$.
	In the case that $\overline{v}$ is the empty list we write
	simply $\mathrm{Tr}(H;\,\, v)$ or $\Tr(H;\,\,v;\,\,r)$.
\end{definition}

For any $k\in \N$ we define an equivalence relation over rooted trees
which generalizes both the relation of "$k$-morphism" as 
defined in \cite{lynch1992probabilities},
and the notion of "$(k,r)$-values" defined in \cite{salvadorbrasil}. 

\begin{definition}\label{def:sim_trees}
	 
	Fix a natural number $k$. We define the
	\textbf{$k$-equivalence} relation over rooted trees, 
	written as $\sim_k$, by induction over their radii as follows:
	
	\begin{itemize}[leftmargin=*]
		\item Any two trees with radius zero are $k$-equivalent.
		Notice that those trees
		consist only of one vertex: their respective roots.
		
		\item Let $r>0$.
		Suppose the $k$-equivalence relation has been
		defined for rooted trees with radius at most $r-1$. Let $\Sigma_{k,r-1}$
		be the set consisting of the $\sim_k$ classes of trees
		with radius at most $r-1$. Let $\rho$ be an special symbol called the
		\textbf{root symbol}. Set $\widehat{\Sigma}_{k,r-1}:=\Sigma_{k,r-1}\cup \{\rho\}$.
		Then a $(k,r)$-\textbf{pattern} is isomorphism class
		of $\widehat{\Sigma}_{k,r-1}$-hypergraphs 
		$(e,\tau)$ that consist of only one edge with no loops and no isolated vertices,
		and satisfy 
		$\tau(v)=\rho$ for exactly one vertex $v\in V(e)$. We
		denote by $P(k,r)$ the set of $(k,r)$-patterns. \par
		Given a rooted tree $(T,v)$ of radius $r$
		we define its \textbf{canonical k-coloring}
		as the map 
		$\tau^k_{(T,v)}: V(T)\rightarrow \widehat{\Sigma}_{k,r-1}$ satisfying that
		$\tau^k_{(T,v)}(u)$ is the
		$\sim_k$ class of $\mathrm{Tr}(T,u;\, \,v)$
		for any $u\neq v$, and $\tau^k_{(T,v)}(v)=\tau$. \par
		Let $T_1$ and $T_2$ be rooted trees of radius $r$. 
		We say that $(T_1,v_1)\sim_k (T_2,v_2)$ 
		if for any pattern $\epsilon\in P(k,r)$	the
		``quantity of initial edges $e_1\in E(T_1)$ such that
		$(e,\tau^k_{(T_\delta,v_\delta)}) \in \epsilon$" 
		and the
		``quantity of initial edges $e_2\in E(T_2)$ such that
		$(e,\tau^k_{(T_\delta,v_\delta)})\in \epsilon$
		" are equal or
		are both greater than $k-1$.
	\end{itemize}
	
\end{definition}

The following is a way of characterizing $\sim_k$ classes
of rooted trees with radii at most $r$ that will be useful later. 

\begin{obs}\label{obs:equivalenttrees}
	Let $\mathbf{T}$ be a $\sim_k$ class of rooted trees with
	radii at most $r$. Then there is a partition $E^1_\mathbf{T},
	E^2_\mathbf{T}$ of $P(k,r)$ and natural numbers $a_\epsilon<k$
	for each $\epsilon\in E^2_\mathbf{T}$ that depends only on 
	$\mathbf{T}$ such that a rooted tree $(T,v)$ belongs to
	$\mathbf{T}$ if and only if the following hold: (1) For any pattern 
	$\epsilon\in E^1_\mathbf{T}$ there are at least $k$ initial edges $e\in E(T)$ such that
	$(e,\tau^k_{(T,v)})\in \epsilon$, and (2) for any pattern 
	$\epsilon\in E^2_\mathbf{T}$ there are exactly
	$a_\epsilon$ initial edges $e\in E(T)$ such that
	$(e,\tau^k_{(T,v)})\in \epsilon$.	
\end{obs}

From this characterization of the $\sim_k$ relation
it follows, by induction over $r$, that the quantity
of $\sim_k$ classes of trees with radii at most $r$ is finite, for any $r\in \N$. \par

\begin{definition} \label{def:sim_general}
	Let $k\in \N$. 
	Given a non-tree connected hypergraph  $H$, we define
	its \textbf{canonical k-coloring} $\tau^k_{H}$
	as the one that assigns to each vertex $v\in V(H)$ the $\sim_k$ 
	class of the tree
	$\Tr(H,v)$.
	Let $H_1$ and $H_2$ be connected hypergraphs which are not trees.
	Set $H^\prime_1:= Center(H_1)$ and $H^\prime_2:= Center(H_2)$.
	We say that $H_1$ and $H_2$ are $k$-equivalent,
	written as $H_1\sim_k H_2$, if
	$( H^\prime_1,\tau^k_{H_1}) \simeq 
	(H^\prime_2,\tau^k_{H_2})$
\end{definition}

\begin{definition} \label{def:agreeability}
	Let $k,r\in\N$ and let $H_1$ and $H_2$ be hypergraphs.
	Let $H^\prime_1:=Core(H_1;r)$ and $H^\prime_2:=Core(H_2;r)$. 
	We say that $H_1$ and $H_2$ are $(k,r)$-agreeable, written
	as $H_1\approx_{k,r} H_2$ if for any $\sim_k$ class $\mathbf{H}$ 
	``the number of connected
	components in $H^\prime_1$ that belong to $\mathbf{H}$" and
	``the number of connected components in $H^\prime_2$ that belong to 
	$\mathbf{H}$" are the same or are both greater than $k-1$.\par

\end{definition}

\begin{definition}
	Let $k,r\in\N$ and let $\Sigma_{(k,r)}$ 
	be the set of $\sim_k$ classes
	of rooted trees with radii at most $r$. Then
	a $(k,r)$-\textbf{cycle} is an isomorphism class
	of $\Sigma_{(k,r)}$-hypergraphs 
	$(H,\tau)$ that are cycles of diameter at most $2r+1$.
	We denote by $C(k,r)$ the set of $(k,r)$-cycles.
\end{definition}

\begin{obs}\label{obs:agreeablecores}
	Let $k,r\in\N$ and let $\mathbf{O}$ be a $\approx_{k,r}$ class
	of $r$-simple hypergraphs. 
	Then there is a partition $U^1_\mathbf{O},
	U^2_\mathbf{O}$ of $C(k,r)$ and natural numbers $a_\omega<k$
	for each $\omega\in U^2_\mathbf{O}$ that depend  only on 
	$\mathbf{O}$ such that a $r$-simple hypergraph $G$ belongs to
	$\mathbf{O}$ if and only if it holds that (1) for any $\omega\in U^1_\mathbf{O}$ 
	there are at least $k$ connected components $H \subset Core(G;r)$ whose cycle 		
	$H^\prime=Center(H)$ satisfies that	$(H^{\prime},\tau^k_{H})\in \omega$, and
	(2) for any $\omega\in U^2_\mathbf{O}$ there are exactly $a_\omega$ 
	connected components $H \subset Core(G;r)$ whose cycle 
	$H^\prime=Center(H)$ satisfies that	$(H^{\prime},\tau^k_{H})\in \omega$.	
\end{obs}

\begin{definition} \label{def:rich}
	Let $H$ be a hypergraph and let $k,r\in\N$. Let
	$X\subset V(H)$ be the set of vertices in $H$
	belonging to some saturated sub-hypergraph of diameter
	at most $2r+1$.	We say that $H$ is $(k,r)$-\textbf{rich}
	if for any $r^\prime\leq r$, vertices $v_1,\dots, v_k$ 
	and $\sim_k$ class $\mathbf{T}$ of trees with radius
	at most $r^\prime$ there exists a vertex $v\in V(H)$
	such that $d(v,X)> 2r^\prime+1$, $d(v,v_i)>2r^\prime+1$ for all
	$v_i$  and $T:=N(v;r^\prime)$ is a tree satisfying
	$(T,v)\in \mathbf{T}$.	
\end{definition}

\subsection{Main result and outline of the proof}
Our goal is to prove the following theorem

\begin{theorem} \label{thm:main}
	Let $\phi$ be a sentence in $FO[\sigma]$. Then
	the function
	$F_\phi: (0,\infty)^{|\sigma|}
	\rightarrow \R$ given by 
	\[
	\InR{\beta} \mapsto \Ln Pr\left( G_n\left(
	\{\beta_R\}_R\right) \models \phi\right)
	\]
	is well defined and analytic. 
\end{theorem}
In fact we prove something stronger. We show that the limit
in last theorem is given by an expression with parameters $\{\beta_R\}_R$
built using rational constants, sums, products and exponentiation with base $e$.
We do so by giving a family of expressions which contains the ones that define 
limit probabilities of FO properties in $G_n(\{\beta\}_R)$.  \par

The main arguments are similar to the ones in the proof of
 \cite[Theorem 2.1]{lynch1992probabilities},
adapted to fit our context. As in that
article the proof is divided into two parts: a model theoretic part and a 
probabilistic part. The main result of the first part is the following

\begin{reptheorem}{thm:Duplicatorwins}
	Let $k\in \N$ and let $H_1$, $H_2$ be hypergraphs. Set $r:=(3^k-1)/2$. Suppose
	that both $H_1$ and $H_2$ are $(k,r)$-rich and
	$H_1\approx_{k,r} H_2$. Then Duplicator wins $\ehr_{k}(H_1,H_2)$
\end{reptheorem}

With regards to the second part, 
the ``landscape'' of $G_n$ can be described similarly to the one 
of $G(n,c/n)$ as in \cite{shelah1994can}: A.a.s for any fixed radius $r$
all neighborhoods $N(v;r)$ in $G_n$ are trees or unicycles, so cycles in 
$G_n$ are far apart. One can find arbitrarily
many copies of any fixed tree, while the expected number of copies 
of any fixed cycle is finite. The main probabilistic results are the
following:
\begin{reptheorem}{cor:simple}
  	Let $r\in \N$. Then a.a.s $G_n$ is $r$-simple.
\end{reptheorem}

\begin{reptheorem}{thm:rich}
	Let $k,r\in \N$. Then a.a.s $G_n$ is $(k,r)$-rich.
\end{reptheorem}

\begin{reptheorem}{thm:agreeabilityprobabilities}
	Let $k,r \in \N$. Let $\mathbf{O}$ be a $\approx_{k,r}$ class
	of $r$-simple hypergraphs. Then 
	\[
	\Ln \mathrm{Pr}\left(G_n\left(\{\beta_R\}_{R\in \sigma}\right)\in \mathbf{O}\right)
	\]
	exists
	and is an analytic expression in $\{\beta_R\}_{R\in \sigma}$.
\end{reptheorem}

A sketch of the proof of \Cref{thm:main} using these results as follows. 
Let $\Phi\in FO[\sigma]$ be a sentence and let $k:=\qr(\Phi)$, $r:=(3^k-1)/2$.
Because of \Cref{thm:Duplicatorwins,thm:rich}
it holds that for any $\approx_{k,r}$ class $\mathbf{O}$
\[
\Ln \mathrm{Pr}\left(G_n \models \Phi\, \big| \, G_n\in \mathbf{O} \right)= \text{ $0$ or $1$ }.
\]
This together with \Cref{cor:simple}
and the fact that there is a finite number of
$\approx_{k,r}$-classes of $r$-simple hypergraphs imply that $\Ln \PR{G_n \models \Phi}$
equals a finite sum of limits of the form $\Ln \PR{G_n \in \mathbf{O}}$,
where $\mathbf{O}$ is some $\approx_{k,r}$-class of $r$-simple hypergraphs.
Finally, using \Cref{thm:agreeabilityprobabilities}
 we get that $\Ln \PR{G_n \models \Phi}$ exists and 
is an analytic expression in $\{\beta_R\}_{R}$, as we wanted.

\section{Model theoretic results}

\subsection{Winning strategies for Duplicator}

During this section $H_1$ and $H_2$ stand for
hypergraphs and $V_1:=V(H_1)$, $V_2:=V(H_2)$.

\begin{definition} \label{def:similar}
Let $\overline{v} \in V_1^*, 
\overline{u} \in V_2^*$ be tuples of the same length.
We write $(H_1,\overline{v})\simeq_{k,r}(H_2, \overline{u})$, if Duplicator wins
$d\ehr_k\left(N(\overline{v};r),
\overline{v};\, N(\overline{u};r),\overline{u}\right)$.
Given $X\subseteq V_1$ and $Y\subseteq V_2$
we write $(H_1,X)\simeq_{k,r} (H_2,Y)$, if we can order $X$, resp. $Y$, to form
lists $\overline{v}$, resp. $\overline{u}$, such that 
$(H_1,\overline{v})\simeq_{k,r}(H_2,\overline{u})$. 
Given $X\in V_1$, $Y\in V_2$ and 
tuples of the same length $\overline{v}\in V_1^*$ and
$\overline{u}\in V_2^*$ we write
 $\left(H_1, (X,\overline{v})  \right)
\simeq_{k,r} \left(H_2, (Y,\overline{u})  \right)$, if
$X$ and $Y$ can be ordered to form lists
$\overline{w}$, resp. $\overline{z}$ such that
$(H_1,\overline{w}^\smallfrown \overline{v})
\simeq_{k,r} (H_2,\overline{z}^\smallfrown \overline{u})$. 
\end{definition}

\begin{definition} \label{def:analogous}
Fix $r\in \N$. Suppose $X\subseteq V_1$ and 
$Y\subseteq V_2$ can
be partitioned into sets $X=X_1\cup \dots \cup X_a$
and $Y=Y_1\cup \dots \cup Y_b$ 
such that all $N(X_i;r)$ and $N(Y_i;r)$
 are connected and disjoint. 
We write $(H_1,X)\cong_{k,r} (H_2,Y)$, 
if for any set $Z\subset V_\delta$, with $\delta\in \{1,2\}$,
among the $X_i$ or the $Y_i$
it is satisfied that ``the number of $X_i$  
such that $(H_\delta, Z) \simeq_{k,r} (H_1,X_i)$" 
and ``the number of $Y_i$ such that
$(H_\delta,Z)\simeq_{k,r} (H_2,Y_i)$"
are both equal or are both greater than $k-1$.
\end{definition}

The main theorem of this section, which
is a strengthening of \cite[Theorem 
2.6.7]{spencer2013strange}, is the following.

\begin{theorem}\label{thm:DuplicatorAux}
	Let $k\in \N$.
	Set $r:=(3^k-1)/2$.
	Suppose there exist
	sets $X\subseteq V_1$, $Y\subseteq V_2$ with the 
	following properties:
	\begin{enumerate}
		\item[(1)] $(H_1,X)\cong_{k,r} (H_2,Y)$.
		\item[(2)]
		\begin{itemize}
			\item Let $r^\prime\leq r$. Let $v\in V_1$ be
			a vertex such that $d(X,v)> 2r^\prime + 1$. Let 
			$\overline{u}\in (V_2)^{k-1}$ be a tuple of vertices. 
			Then there exists $u\in V_2$ such that 
			$d(u,\overline{u})>2r^\prime+1$,
			$d(Y,u)>2r^\prime +1$ and
			$(H_1,v)\simeq_{k,r^\prime} (H_2,u)$.	
			\item Let $r^\prime\leq r$. Let $u\in V_2$ be
			a vertex such that $d(Y,u)> 2r^\prime + 1$. Let 
			$\overline{v}\in (V_1)^{k-1}$ be a tuple of vertices. 
			Then there exists $v\in V_1$ such that 
			$d(v,\overline{v})>2r^\prime+1$,
			$d(X,v)>2r^\prime +1$ and
			$(H_1,v)\simeq_{k,r^\prime} (H_2,u)$
		\end{itemize}
	\end{enumerate}
	Then Duplicator wins $\ehr_k\left(H_1;H_2\right)$.
\end{theorem}

In order to prove this theorem we need to make two observations
and prove a previous lemma. 

\begin{obs} \label{obs1}
	Let $k\in \N$ and let
	$\overline{v}\in V(H_1)^*$, $\overline{u}\in V(H_2)^*$ be of equal length. Suppose
	Duplicator wins $d\ehr_k(H_1,\overline{v}; \, H_2,\overline{u})$.
	Then, for any $r\in \N$, $(H_1, \overline{v})\simeq_{k,r} 
	(H_2,\overline{u})$. 
\end{obs}

\begin{obs} \label{obs2}
	Let $k\in \N$ and let
	$\overline{v}\in V(H_1)^*$, $\overline{u}\in V(H_2)^*$ be of equal length. Suppose 
	Duplicator wins $d\ehr_k(H_1,\overline{v};\,H_2,\overline{u})$. 
	Let $v\in V(H_1),u\in V(H_2)$ be the vertices
	played in the first round of an instance of the game 
	where Duplicator is following a winning strategy. Then 
	Duplicator also wins $d\ehr_{k-1}(H_1,\overline{v_2}; \,
	H_2,\overline{u_2})$, where $\overline{v_2}:=\overline{v}^\smallfrown v$
	and $\overline{u_2}:=\overline{u}^\smallfrown u$.
\end{obs}

\begin{lemma} \label{lemm:Duplicator}
	Let $k,r\in \N$. Let $\overline{v}\in V_1^*$ and
	$\overline{u} \in V_2^*$ be of equal length.
	$(H_1,\overline{v})\simeq_{k,3r+1} (H_2,\overline{u})$.
	Let $v \in V_1$ and $u\in V_2$
	be vertices played in the first round of an instance of 
	\[
	d\ehr_k\left(\, N(\overline{v};3r+1),
	\overline{v}; \quad N(\overline{u};3r+1),\overline{u}\, \right)
	\]
	where Duplicator is following a winning strategy. Further suppose
	that $d(\overline{v},v)\leq 2r+1$ (and in consequence
	$d(\overline{u},u)\leq 2r+1$ as well). 
	Let $\overline{v_2}:=\overline{v}^\smallfrown v$
	and $\overline{u_2}:=\overline{u}\smallfrown u$.
	Then $(H_1,\overline{v_2})\simeq_{k-1,r}
	(H_2,\overline{u_2})$.
\end{lemma}

\begin{proof}
	Using \Cref{obs2} we get that Duplicator wins 
	\[
	d\ehr_{k-1}\left(\, N(\overline{v};3r+1),
	\overline{v_2}; \quad N(\overline{u};3r+1)
	,\overline{u_2} \, \right)\]
	as well. Call $H^\prime_1=N(\overline{v};3r+1)$,
	$H^\prime_2=N(\overline{u};3r+1)$. Then by \Cref{obs2}
	Duplicator wins
	\[
	d\ehr_{k-1}\left(\, N^{H^\prime_1}(
	\overline{v_2};r),\overline{v_2};\quad
	N^{H^\prime_2}(\overline{u_2};r),\overline{u_2}\, \right).
	\]
	Because of this if we prove $N^{H_1}(\overline{v_2};r)
	=N^{H^\prime_1}(\overline{v_2};r)$ and $N^{H_2}(\overline{u_2};r)
	=N^{H^\prime_2}(\overline{u_2};r)$, then we are finished. 
	Let $z\in N^{H_1}(v^\prime;r)$. Then
	$d(z,\overline{v})\leq d(z,v^\prime)+d(v^\prime,\overline{v})=3r+1$.
	As a consequence, $N^{H_1}(v;r)\subset H^\prime_1$. Thus,
	$N^{H_1}(\overline{v_2};r)\subseteq H^\prime_1$, and $N^{H_1}(\overline{v_2};r)
	=N^{H^\prime_1}(\overline{v_2};r)$. Analogously we obtain 
	$N^{H_2}(\overline{u_2};r)=N^{H^\prime_2}(\overline{u_2};r)$, as we wanted. 
\end{proof}

\begin{proof}[Proof of \Cref{thm:DuplicatorAux}]

	Let $X_1,\dots,X_a$ and 
	$Y_1,\dots,Y_b$ be partitions of 
	$X$ and $Y$ respectively
	as in the definition of $\cong_{k,r}$.
	Let $r_0:=(3^k-1)/2$ and $r_i:=(r_{i-1}-1)/3$ for each
	$1\leq i \leq k$. 
	Let $v^1_i$ and $v^2_i$ be the vertices played
	in $H_1$ and $H_2$ respectively during the $i$-th
	round of $\ehr_k(H_1,H_2)$. 
	We show a winning strategy for Duplicator
	in $\ehr_{k}\left(H_1;\, H_2\right)$. For each $0\leq i \leq k$,
	Duplicator will	keep track of some marked sets 
	of vertices $T\subset V_1$, $S\subset V_2$. 
	For $\delta=1,2$ each marked set
	$T\subset V_\delta$ will have associated a tuple
	of vertices $\overline{v}(T)\in V_\delta^*$ consisting
	of the vertices played in $H_\delta$ so far that were 
	"appropriately close" to $T$ when chosen,  ordered according
	to the rounds they where played in.
	The game will start with no sets of vertices marked and 
	at the end of the $i$-th
	round Duplicator will perform one of the two
	following operations:
	\begin{itemize}
	\item Mark two sets $S\subset V_1$ and $T\subset V_2$ and
	define $\overline{v}(S):=v^1_i$ and $\overline{v}(T):=v^2_i$.
	\item Given two sets $S\subset V_1$, $T\subset V_2$ that were 
	previously marked during the same round, append $v^1_i$
	and $v^2_i$ to $\overline{v}(S)$ and $\overline{v}(T)$ 
	respectively. 
	\end{itemize}
	We show that Duplicator can play in such a way that at the
	end round the following are satisfied:
	\begin{itemize}
		\item[(i)] For $\delta=1,2$, each vertex played
		so far $v^\delta_j\in V_\delta$ belongs to 
		$\overline{v}(S)$ for a	 unique marked set 
		$S\subset V_\delta$.
		\item[(ii)] Let $S\subset V_1$ and $T\subset V_2$ be sets
		marked during the same round. Then any previously played
		vertex $v^1_j$ occupies a position in $\overline{v}(S)$
		if and only if $v^2_j$ occupies the same position 
		in $\overline{v}(T)$.
		\item[(iii)] 
			\begin{itemize}
			\item  Let $S\subset V_1$ be a marked set. Then 
			for any different marked $S^\prime \subset V_1$ 
			of any different $S^\prime$ among $X_1,\dots,X_a$
			it holds $d(S,S^\prime)>2r_i + 1$. 
			\item Let $T\subset V_2$ be a marked set. Then
			for any different marked $T^\prime \subset V_2$
			or any different $T^\prime$ among
			$Y_1,\dots, Y_b$  it holds $d(T,T^\prime)>2r_i +1$.
			\end{itemize}
		\item[(iv)] Let $S\subset V_1, T\subset V_2$ be sets
		marked during the same round. Then
		\[
		\left(H_1, (S,\overline{v}(S))\right)\simeq_{k-i,r_i}
		\left(H_2, (T,\overline{v}(T))\right).	\]

	\end{itemize}
	In particular, if conditions (i) to (iv) are satisfied this means
	that if $\overline{v}^1:=(v^1_1,\dots,v^1_i)$ and
	$\overline{v}^2:=(v^2_1,\dots, v^2_i)$ are the vertices played so far
	then Duplicator wins
	\[
	d\ehr_{k-i}\left(\,
	N(\overline{v}^1;\, r_i), \overline{v}^1; \quad
	N(\overline{v}^2;\, r_i),\overline{v}^2 \,
	\right),	
	\]
	And at the end of the $k$-th round Duplicator will have won
	$\ehr(H_1;\, H_2)$. \par
	The game $d\ehr_k(H_1; \, H_2)$ proceeds as follows.
	Clearly properties (i) to (iv) hold at the beginning of 
	the game.  Suppose that
	Duplicator can play in such a way that properties (i) to (iv) hold 
	until the beginning of the $i$-th round. Suppose
	during the $i$-th round Spoiler chooses $v^1_i\in V_1$ 
	(the case where they play in $V_2$ is symmetric). There are 
	three possible cases:
	\begin{itemize}[leftmargin=*]
		\item For some unique previously marked set $S\subset V_1$ 
		we have $d(S\cup \overline{v},\, v^1_i)\leq 2r_i +1$. 
		In this case let $T\subset V_2$ be the set in $H_2$ marked in the 
		same round as $T$. By hypothesis 
		\[ \left(H_1,(S,\overline{v}(S))\right)
		\simeq_{k-i+1,3r_i+1}
		\left(H_2,(T,\overline{v}(T))\right). 		
		\]
		Then, by definition,
		for some orderings $\overline{w}$, $\overline{z}$
		of the vertices in $S$ and $T$ respectively it holds that
		Duplicator wins
		\[
		d\ehr_{k-i+1}\left(\,
		N(\overline{w}^\smallfrown \overline{v}(S); \, 3r_i	+ 1)
		, \overline{w}^\smallfrown \overline{v}(S); \quad
		N(\overline{z}^\smallfrown \overline{v}(T); \, 3r_i	+ 1)
		, \overline{z}^\smallfrown \overline{v}(T)\,	
		\right).		
		\]
		Thus Duplicator can choose $v^2_i\in V_2$ according to the 
		winning strategy in that game. After this Duplicator sets 
		$\overline{v}(S):= \overline{v}(S)^\smallfrown v^1_i$, and
		$\overline{v}(T):= \overline{v}(T)^\smallfrown v^2_i$. Notice
		that because of \Cref{lemm:Duplicator} now
		\[
		\left(H_1, (S,\overline{v}(S))\right)\simeq_{k-i,r_i}
		\left(H_2, (T,\overline{v}(T))\right).
		\]
		\item For all marked sets $S\subset V_1$ it holds
		$d(S\cup \overline{v}(S), \quad v^1_i)>2r_i +1$, but there is
		a unique $S$ among $X_1,\dots, X_a$ such that
		$d(S,v^1_i)\leq 2r_i+1$. In this case from condition (1)
		of the statement follows that there is some non-marked
		set $T$ among $Y_1,\dots, Y_b$ such that
		\[
		(H_1,S)\simeq_{k-i+1,3r_i+1} (H_2,T).\] 
		Thus, by definition, for some orderings $\overline{w}$, 
		$\overline{z}$ of the vertices in $S$ and $T$ respectively,
		Duplicator wins
		\[
		d\ehr_{k-i+1}\left( \,
		N(\overline{w};3r_i+1), \overline{w};
		\quad
		N(\overline{z};3r_i+1), \overline{z}
		\, \right).		
		\]
		Then Duplicator can choose $v^2_i\in V_2$ according 
		to a winning strategy for this game. After this Duplicator
		marks both $S$ and $T$ and sets 
		$\overline{v}(S):=v^1_i$, and 
		$\overline{v}(T):=v^2_i$. Notice
		that because of \Cref{lemm:Duplicator} now
		\[
		\left(H_1, (S,\overline{v}(S))\right)\simeq_{k-i,r_i}
		\left(H_2, (T,\overline{v}(T))\right).
		\]
		\item For all marked sets $S\subset V_1$ we have
		$d(S\cup \overline{v}(S), \, v^1_i)>2r_i +1$, and
		for all sets $S$ among $X_1,\dots, X_a$ it also holds
		$d(S,v^1_i)> 2r_i+1$. In this case
		from condition (2) of the statement it follows that
		Duplicator can choose
		$v^2_i\in V_2$ such that
		(A)
		$d(T\cup \overline{v}(T),\, v^2_i)>2r_i+1$ for all
		marked sets $T\subset V_2$, (B)
		$d(T, v^2_i)> 2r_i+1$ for all sets $T$ among
		$Y_1,\dots, Y_b$, and (C) 
		$(H_1,v^1_i)\simeq_{k-i,r_i} (H_2, v^2_i)$.
		After this Duplicator marks both $S=\{
		v^1_i\}$ and $T=\{v^2_i\}$ and sets
		$\overline{v}(S):=v^1_i$, and
		$\overline{v}(T):=v^2_i$. 		
	\end{itemize}	 
	The fact that conditions
	(i) to (iv) still hold at the end of the round
	follows from comparing $r_{i-1}$ and $r_{i}$ as well 
	as applying \Cref{obs1} and \Cref{obs2}.

\end{proof}

\subsection{k-Equivalent trees} \label{sect:equivtrees}

We want prove the following.

\begin{theorem} \label{thm:equivalenttrees} 
	Let $k\in \N$. 
	Let $(T_1,v_1)$ and $(T_2,v_2)$ be rooted trees such
	that $(T_1,v_1)\sim_k (T_2,v_2)$.
	Then Duplicator wins
	$d\ehr_{k}(T_1,v_1;\, \, T_2,v_2)$.
\end{theorem}

Before proceeding with the proof we need  an auxiliary
result. Let $(T,v)$ be a rooted tree and $e$ an 
initial edge of $T$. We define $\mathrm{Tr}(T,v;\, e)$ as
the induced tree $T[X]$ on the set
$X:=\{v\} \cup \{\, u\in V(T) \, | \, d(v,u) = 1 + d(e,u) \,\}$,
with $v$ as the root. In other words, 
$\mathrm{Tr}\left(T,v;\, e\right)$ is the tree consisting of 
$v$ and all the vertices
in $T$ whose only path to $v$ contains $e$. 

\begin{lemma} \label{lem:equivalentedges}
	Let $k\in \N$ and fix $r>0$. Suppose theorem \ref{thm:equivalenttrees}
	holds for rooted trees with radii at most $r$.
	Let $(T_1,v_1)$ and $(T_2,v_2)$ be rooted trees with radius
	$r+1$. Let $\tau^k_{(T_1,v_1)}$ and $\tau^k_{(T_2,v_2)}$
	be colorings over $T_1$ and $T_2$ as in \Cref{def:sim_trees}
	Let $e_1$ and $e_2$ be initial edges 
	of $T_1$ and $T_2$ respectively satisfying 
	$(e_1,\tau^k_{(T_1,v_1)})\simeq (e_2,\tau^k_{(T_2,v_2)})$. Name 
	$T^\prime_1:=\mathrm{Tr}(T_1,v_1; \,\,e_1)$ and 
	$T^\prime_2:=\mathrm{Tr}(T_2,v_2;\,\,e_2)$. Then
	Duplicator wins 
	$d\ehr_{k}(T^\prime_1, v_1; \, \, T^\prime_2,v_2)$.
\end{lemma}
\begin{proof}
	We show a winning strategy for Duplicator.
	At the beginning of the game fix an isomorphism 
	$f:V(e_1)\rightarrow V(e_2)$ between 
	$(e_1,\tau^k_{(T_1,v_1)})$
	and $(e_2,\tau^k_{(T_2,v_2)})$.
	Suppose in the $i$-th round of the game Spoiler
	plays on $T^\prime_1$. The other case is symmetric.
	If Spoiler plays $v_1$ then Duplicator chooses $v_2$. 
	Otherwise, Spoiler plays a vertex $v$ that belongs
	to some $\mathrm{Tr}(T^\prime_1,v_1;\,\, u)$ for a unique $u\in V(e_1)$
	different from the root $v_1$. 
	Set $T^{\prime\prime}_1:=
	\mathrm{Tr}\left(T^\prime_1,v_1;\,\, u\right)$
	and
	$T^{\prime\prime}_2:=\mathrm{Tr}\left(T^\prime_2,v_2;\,\, f(u)\right)$
	Then, as $\tau^k_{(T_1,v_1)}\left(u\right)=\tau^k_{(T_2,v_2)}\left(f(u)\right)$,
	we obtain
	$\left(T^{\prime\prime}_1,u\right) \sim_k 
	\left(T^{\prime \prime}_2,f(u)\right)$.
	As both these trees have radii at most $r$, 
	by assumption Duplicator has a winning 
	strategy in
	$d\ehr_{k}\left(\,T^{\prime\prime}_1, u;
	\quad T^{\prime \prime}_2, f(u)\,  \right)$		
	and they can follow it considering the previous plays in
	$T^{\prime\prime}_1$ and $T^{\prime\prime}_2$.	
\end{proof}

\begin{proof}[Proof of \Cref{thm:equivalenttrees}]~ \par
	Notice that, as $(T_1,v_1)\morph{k} (T_2,v_2)$, both $T_1$ and
	$T_2$ have the same radius $r$.
	We prove the result by induction on $r$.
	If $r=0$ then both $T_1$ and $T_2$ consist
	of only one vertex and we are done.
	Now let $r>0$ and assume that the 
	statement is true for all smaller values of $r$.
	Let $\tau^k_{(T_1,v_1)}$ and $\tau^k_{(T_2,v_2)}$ 
	be the colorings over $T_1$ and $T_2$ as in 
	\Cref{def:sim_trees}. 
	We show that there is a winning strategy 
	for Duplicator in
	$d\ehr_k(T_1,v_1;\,\, T_2,v_2)$.
	At the start of the game, set all the initial edges
	in $T_1$ and $T_2$ as non-marked. 
	Suppose in the $i$-th round Spoiler plays in 
	$T_1$. The other case is symmetric. 
	If Spoiler plays $v_1$ then Duplicator plays $v_2$.
	Otherwise, the vertex played by Spoiler belongs to
	$\mathrm{Tr}(T_1,v_1;\,\,e_1)$
	for a unique initial edge $e_1$ of $T_1$. 
	There are two possibilities:
	\begin{itemize}[leftmargin=*]
		\item If $e_1$ is not marked yet, mark it. 
		In this case, there is a 
		non-marked initial
		edge $e_2$ in $T_2$ satisfying 
		$\left(e_1,\tau^k_{(T_1,v_1)}\right)\simeq
		\left(e_2,\tau^k_{(T_2,v_2)}  \right)$.
		Mark $e_2$ as well. 
		Set $T^\prime_1:=\mathrm{Tr}(T_1,v_1;\,\,e_1)$
		and
		$T^\prime_2:=\mathrm{Tr}(T_2,v_2;\,\,e_2)$
		Because of
		\Cref{lem:equivalentedges}, Duplicator
		has a winning strategy in
		$ d\ehr{k}(T^\prime_1, v_1;$ $\,\, T^\prime_2,v_2)$
		and can play according to it.
		\item If $e_1$ is already marked then there is
		a unique initial edge $e_2$ in $T_2$ that was 
		marked during the same round as $e_1$ and it satisfies 
		$\left(e_1,\tau^k_{(T_1,v_1)}\right)\simeq
		\left(e_2,\tau^k_{(T_2,v_2)}  \right)$.	
		Again, because of \Cref{lem:equivalentedges}, 
		Duplicator has a winning strategy in
		$d\ehr{k}(T^\prime_1, v_1; \,\, T^\prime_2,v_2)$
		and can continue playing according to it taking
		into account the plays made previously in 
		$T^\prime_1$ and $T^\prime_2$.	
	\end{itemize}
\vspace{-2em}
\end{proof}

\subsection{k-Equivalent hypergraphs} \label{sect:equivunicycles}
	
	\begin{theorem} \label{thm:strategyaux}
		Let $H_1$ and $H_2$ be non-tree 
		connected hypergraphs satisfying
		$H_1\sim_k H_2$. 
		Set $H^\prime_1:= Center(H_1)$ and 
		$H^\prime_2:= Center(H_2)$. Let 
		$\tau^k_{H_1}, \tau^k_{H_2}$ be as in
		\Cref{def:sim_general}.
		Let $f$ be an isomorphism
		between $( H^\prime_1,\tau^k_{H_1})$ and 
		$(H^\prime_2,\tau^k_{H_2})$. Let $\overline{v}$ 
		be an ordering of the vertices of $H^\prime_1$ and let
		$\overline{u}:=f(\overline{v})$ be the corresponding
		ordering of the vertices of $H^\prime_2$. Then
		Duplicator wins 
		$	d\ehr_{k}\left(\,
		H^\prime_1,\overline{v};\,\,
		H^\prime_2,\overline{u}\,
		\right).
		$
	\end{theorem}
	\begin{proof}
		The winning strategy for Duplicator is as follows. 
		Suppose at the beginning of the $i$-th round
		Spoiler plays in $H_1$ (the case where they play in
		$H_2$ is symmetric). Then Spoiler has chosen a vertex 
		that belongs to $\mathrm{Tr}(H_1;\,\,u)$ for a unique
		$u\in H^\prime_1$. 
		Set $T_1:=\mathrm{Tr}\left(H_1;\,\,u\right)$ and
		$T_2:=\mathrm{Tr}\left(H_2;\,\,f(u)\right)$.
		By hypothesis
		$(T_1,u)\sim_k (T_2,f(u))$. 
		Then because of \Cref{thm:equivalenttrees} we have that
		Duplicator has a winning strategy in
		$
		d\ehr_{k}\left(\,
		T_1,u; \, \, T_2, f(u)\,\right),
		$
		and they can follow it taking into account the previous
		moves made in $T_1$ and $T_2$, if any. In particular, 
		if Spoiler has chosen
		$u$ then Duplicator will necessarily choose $f(u)$.
		One can easily check that distances are preserved
		following this strategy. 
	\end{proof}
		
\subsection{Main result} \label{sec:Core}

\begin{lemma} \label{lem:aux1}
	Let $k,r\in \N$ and let $H_1, H_2$ be hypergraphs such that 
	$H_1\approx_{k,r} H_2$. Let $X$ and $Y$ be the
	sets of vertices in $H_1$, resp. $H_2$, that
	belong to a saturated sub-hypergraph of diameter 
	at most $2r+1$. Then $(H_1,X)\cong_{k,r} (H_2,Y)$ in the sense of
	\Cref{def:analogous}.
\end{lemma}
\begin{proof}
	Let $X_1,\dots, X_a$ and $Y_1,\dots, Y_b$ be partitions of
	$X$ and $Y$ such that each $N(X_i;r)$ and $N(Y_i;r)$ is 
	a connected component of $Core(H_1;r)$, resp. $Core(H_2;r)$.
	Because of \Cref{thm:strategyaux}
	$N(X_i;r)\sim_k N(Y_j;r)$ implies 
	$(H_1,X_i)\simeq_{k,r} (H_2,Y_j)$ in the sense of
	\Cref{def:similar}. The result follows now from
	 the definition of $H_1\approx_{k,r}H_2$.	
\end{proof}

\begin{theorem}\label{thm:Duplicatorwins}
	Let $k\in \N$, and set $r:=(3^k-1)/2$.
	Let $H_1$, $H_2$ be hypergraphs.
	Suppose
	that both $H_1$ and $H_2$ are $(k,r)$-rich and
	$H_1\approx_{k,r} H_2$. Then Duplicator wins $\ehr_{k}(H_1,H_2)$.
\end{theorem}
	\begin{proof}
		Because of the previous lemma we can apply 
		\Cref{thm:DuplicatorAux} with $X\subset V(H_1)$ 
		and	$Y\subset V(H_2)$ defined as before. The
		hypothesis of $(k,r)$-richness on both $H_1, H_2$ ensures that
		condition (2) in the statement of \Cref{thm:DuplicatorAux}
		holds. 
	\end{proof}

\section{Probabilistic results}

\subsection{Almost all hypergraphs are simple}

\begin{lemma}
	Let $H$ be a hypergraph, and let $X_n$ be the 
	random variable equal to the number of copies of $H$ in 
	$G_n$. Then 
	$\mathrm{E}\big[X_n\big]=\Theta(n^{-\ex(H)})$.  
\end{lemma}
\begin{proof}
We have
\[
 \mathrm{E}\big[X_n\big]=
 \sum_{H^\prime \in Copies(H,[n])} \PR{H^\prime \subset G_n}.
\]	
We also have that $\Big|Copies(H,[n])\Big|=\frac{(n)_{|H|}}{\aut(H)}$. Also, 
for any $H^\prime \in Copies(H,[n])$ it holds that
\[
\PR{H^\prime \subset G_n}\sim \prod_{R\in \sigma} \left(\frac{\beta_R}{n^{ar(R)-1}} 
\right)^{|E_R(H)|}.
\]
Substituting in the first equation we get
\[
\mathrm{E}\big[X_n\big]\sim 
\frac{(n)_{|H|}}{\aut(H)}
\prod_{R\in \sigma} \left(\frac{\beta_R}{n^{ar(R)-1}}\right)^{|E_R(H)|}
\sim
n^{-\ex(H)}  \frac{\prod_{R\in \sigma} \beta_R^{
|E_R{H}| }}{\aut(H)}.
\]	
\end{proof}

\begin{lemma} \label{lem:nocopiesdense}
	Let $H$ be a hypergraph such that $\ex(H)>0$. Then
	a.a.s there are no copies of $H$ in $G_n$. 
\end{lemma}  
\begin{proof}
	Because of the previous lemma
	$\mathrm{E}\big[\# \text{ copies of }H \text{ in } G_n\big] 
	\xrightarrow[]{n\to \infty} 0$ . An application of the first moment
	method yields the desired result. 
\end{proof} 

\begin{lemma}\label{lem:nocopiesfixed}
	Let $H$ be a hypergraph. Let
	$\overline{v}\in (\N)_*$ be a list of vertices
	with $\len(\overline{v})\leq |V(H)|$.
	For each $n\in \N$ 
	let $X_n$ be the random variable that
	counts the copies of $H$ in $G_n$ that contain the vertices
	in $\overline{v}$. Then
	$
	\mathrm{E}\big[ X_n \big]=\Theta(n^{-\ex(H)-\len(\overline{v})})$.
\end{lemma}
\begin{proof}
	The number of hypergraphs $H^\prime \in Copies(H,[n])$
	that contain all vertices in $\overline{v}$ is asymptotically
	$\sim n^{|V(H)|-\len(\overline{v})}$ for
	some constant $C$. Then,
	\[
	\mathrm{E}\big[ X_n \big]\sim 
	C   n^{|V(H)|-\len(\overline{v})}
	 \prod_{R\in\tau} \left( \frac{\beta_R}{n^{ar(R)-1}}\right)^{e_R(H)}=
	n^{-\ex(H)-\len(\overline{v})}  C  
	\prod_{R\in\tau} \left( \beta_R \right)^{e_R(H)}.
	\]
\end{proof}

	Given a hypergraph $H$ and an edge $e\in E(H)$ we
	define the operation of \textbf{cutting} the edge 
	$e$ as removing $e$ from $H$ and then removing any isolated
	vertices from the resulting hypergraph. \par

\begin{lemma}
	Let $G$ be a
	dense hypergraph with diameter at most $r$,
	and let $H\subset G$ be a connected 
	sub-hypergraph with $\ex(H)<\ex(G)$. Then
	there is a connected
	sub-hypergraph $H^\prime \subset G$
	satisfying $H\subset H^\prime$,
	$\ex(H)<\ex(H^\prime)$ and that 
	$|E(H^\prime)|\leq |E(H)|+2  r + 1$, 
\end{lemma}
\begin{proof}
	Suppose there is some edge $e\in E(G)\setminus E(H)$ with
	and $\ex(e)\geq 0$. Let $P$ be a path
	of length at most $r$ joining $H$ and $e$ in $G$. 
	Then $H^\prime:=H\cup P \cup e$ satisfies the conditions of 
	the statement. Otherwise, all edges $e\in E(G)\setminus E(H)$ 
	satisfy $\ex(e)=-1$. In this case we successively cut
	edges $e$ from $G$ such that $d(e, H)$ is the maximum possible
	(notice that this always yields a connected hypergraph)
	until we obtain a hypergraph $G^\prime$ with $\ex(G^\prime)<\ex(G)$.
	Let $e$ be the edge that was cut last. Then $V(G^\prime)\cap V(e)=
	\ex(G)-\ex(G^\prime)+1 \geq 2$. Let $v_1, v_2\in V(G^\prime)\cap V(e)$,
	and let $P_1$, $P_2$ be paths of length at most $r$ that join $H$
	with $v_1$ and $v_2$ respectively in $G^\prime$. Then the hypergraph
	$H^\prime:=H\cup e \cup P^1 \cup P^2$ satisfies the conditions 
	in the statement. 
\end{proof}
\begin{lemma}
	Let $G$ be a
	dense hypergraph of diameter at most 
	$r$. Then $G$ contains a connected 
	dense sub-hypergraph $H$ 
	with $|E(H)|\leq 4r+2$. 
\end{lemma}
\begin{proof}
	Apply the previous lemma twice starting with $G$ and taking
	as $H$ a sub-hypergraph of $G$ consisting of a single vertex and 
	no edges. 
\end{proof}

In particular, if we define $l:=\max\limits_{R\in \sigma} \,\, ar(R)$
the last lemma implies that, if $G$ is a dense hypergraph whose
diameter is at most $r$ then $G$ contains a dense sub-hypergraph
$H$ with $|H|\leq l  (4r+2)$.\par

\begin{theorem} \label{thm:sparse}
	Let $r\in \N$. Then a.a.s $G_n$ is $r$-sparse. 
\end{theorem}
\begin{proof}
	Because of the last lemma there is a constant $R$ such that 
	``$G$ does not contain dense hypergraphs of size bounded by $R$" implies
	that ``$G$ is $r$-sparse". Thus,
	\[ \Ln \mathrm{Pr}\left( G_n \text{ is } r \text{-sparse}  \right)
	\geq \Ln \mathrm{Pr} \left( G_n \text{ does not contain dense 
	hypergraphs of size} \leq R\right).\] 
	Because of	\Cref{lem:nocopiesdense}, given a fixed dense hypergraph,
	the probability that $G_n$ contains no copies
	of it tends to $1$ as $n$ goes to infinity. Using that
	there are a finite number of $\sim$ classes of dense hypergraphs whose
	size bounded by	$R$, we deduce that the RHS of the last inequality tends to $1$. 
\end{proof}
As a corollary we obtain the needed result. 
\begin{theorem}\label{cor:simple}
	Let $r\in \N$. Then a.a.s $G_n$ is $r$-simple.
\end{theorem}
\begin{proof}
	If some connected component of $Core(G_n;r)$ is not a cycle then either $G_n$ contains a dense hypergraph of diameter at most $4r+1$,
	or $G_n$ contains two cycles of diameter at most $2r+1$ that are at
	distance at most $2r+1$. In the second case, considering the two cycles
	and the path joining them, $G_n$ contains a dense hypergraph of diameter
	bounded by $6r+3$. Hence the fact that $G_n$ is $(6r+3)$-sparse
	implies that $G_n$ is $r$-simple. Because of the previous theorem 
	$G_n$ is a.a.s $(6r+3)$-sparse and the result follows. 
\end{proof}

\begin{lemma} \label{lem:disjointtrees}
	Let $\overline{v}\in (\N)_*$ and let $r\in \N$. Then
	a.a.s, for all vertices $v\in \overline{v}$ the neighborhoods 
	$N(v;r)$ are all trees and they are all disjoint. 
	
	\begin{proof}
		An application of the first moment method together with
		\Cref{lem:nocopiesfixed} and the fact that there is a finite number
		of $\simeq$ classes of paths whose length is at most $2r+1$, implies that
		a.a.s the $N(v;r)$ are disjoint. 
		Also, because of \Cref{thm:sparse} a.a.s the $N(v;r)$ are either 
		trees or unicycles. But if any of the $N(v;r)$  was an unicycle then
		in $G_n$ there would exist a path $P$ of length at most
		$2r+1$ joining some vertex $v\in \overline{v}$
		with a cycle $C$ of diameter at most $2r+1$. Using \Cref{lem:nocopiesfixed}
		again, as well as the fact that there is a finite number of possible $\simeq$
		classes for
		$P\cup C$, we obtain that a.a.s no such $P$ and $C$ exist. 
		In consequence all the $N(v;r)$  are disjoint trees as we wanted to
		prove. \end{proof}
\end{lemma}

\begin{lemma}\label{lem:far_away}
	Let $\overline{v} \subset \N*$ be a finite set of fixed vertices and let 
	$\pi(\overline{x})$ be an edge sentence such that
	$\len(\overline{x})=\len(\overline{v})$. 
	Define $G_n^\prime=G_n \setminus E[\overline{v}]$ (i.e. $G_n$ minus all the
	edges induced on $\overline{v}$). Fix $r\in \N$. 
	Then a.a.s for all vertices $v\in \overline{v}$ the neighborhoods
	$N^{G^\prime_n}(v;\,r)$  are disjoint trees.  
\end{lemma}
\begin{proof}
	Let $A_n$ be the event that the $N^{G^\prime_n}(v;\,r)$  are disjoint trees.  
	Notice that $A_n$ does not concern the possible edges
	induced over $\overline{v}$. Because edges are independent
	in our random model, we have that
	$\mathrm{Pr}\left(A_n \, | \, \pi(\overline{v})\right)
	=\mathrm{Pr}(A_n)$. Now the result follows from 
	\Cref{lem:disjointtrees} using that $G^\prime_n\subset G_n$.  
\end{proof}

\subsection{Probabilities of trees}

\begin{definition} \label{def:treeprobabilies}
We define $\Lambda$ and $M$ as the minimal families
of expressions with arguments $\{\beta_R\}_{R\in\sigma}$
that satisfy the conditions: \textbf{(1)}
$1\in \Lambda$, \textbf{(2)} 
for any $R\in \sigma$, any positive $b\in \N$,
and $\overline{\lambda} \in \Lambda^*$,
the expression $(\beta_R/b) \prod_{\lambda\in \overline{\lambda}}
\lambda$
belongs to $M$, \textbf{(3)}
for any $\mu\in M$ and any $n\in \N$ both
$\mathrm{Poiss}_{\mu}(n)$ and $\mathrm{Poiss}_\mu(\geq n)$ are in $\Lambda$, 
and  \textbf{(4)} for any $\lambda_1,\lambda_2 \in \Lambda$, the
product $\lambda_1\lambda_2$ belongs to $\Lambda$ as well.
\end{definition}

\begin{definition}
	Let $r\in \N$ and let $\mathbf{T}$ be a $\sim_k$ class
	of trees with radius at most $r$. Let $v\in \N$ be an arbitrary 
	vertex. We define $\mathrm{Pr}[r,\mathbf{T}]$ as the limit
	\[
	\Ln \mathrm{Pr}\left(
	Tr(G_n,\,v;\,v;\,r)\in \mathbf{T}\right).
	\]
\end{definition}

Note that the 
definition of  $\mathrm{Pr}[r,\mathbf{T}]$ does not depend on the
choice of $v$. The goal of this section is to show 
that $\mathrm{Pr}[r,\mathbf{T}]$ exists and is
an expression with parameters
$\{\beta_R\}_{R\in\sigma}$ belonging to $\Lambda$ 
for any choice of $r$ and $\mathbf{T}$. \par

\begin{theorem} \label{thm:BigTrees}
	Fix $r\in \N$. Let $k\in \N$ The following hold:
	\begin{itemize}
		\item[(1)] Let $\mathbf{T}$ be a
		$k$-equivalence class of trees with radii at most $r$.	Then 
		$
		\mathrm{Pr}[r,\mathbf{T}]
		$
		exists,
		is positive for all choices of 
		$\{\beta_R\}_{R}\in (0,\infty)^{|\sigma|}$,
		and is an expression
		in $\Lambda$.
		\item[(2)] Let $\overline{u}\in (\N)_*$,
		and let $\pi(\overline{x})\in FO[\sigma]$ be a consistent
		edge sentence such that 
		$\len(\overline{x})=\len(\overline{u})$.
		Let $\overline{v}\in (\N)_*$ be vertices contained
		in $\overline{u}$. For each $v\in \overline{v}$
		let $\mathbf{T}_v$ be a $k$-equivalence class
		of trees with radii	at most $r$. Then
		\[
		\Ln \mathrm{Pr}\left( \bigwedge_{v\in \overline{v}} 
		Tr\left(G_n, \overline{u};\,\,v;\,\,r\right)\in \mathbf{T}_v 
		\, | \, \pi(\overline{u})
		\right)= \prod_{v\in \overline{v}} \mathrm{Pr}[r,\mathbf{T}_v]. \]	 	
	\end{itemize}
\end{theorem}
We devote the rest of this section to proving this
theorem. The proof is by induction on $r$. 
Recall that	all trees with radius zero are $k$-equivalent. Thus,
the limits appearing in conditions (1) and (2) are both equal to $1$
in the case $r=0$. 

\begin{lemma}
	Conditions (1) and (2) of \Cref{thm:BigTrees} 
	are satisfied for $r=0$.
\end{lemma}

\begin{definition}
	Let $k\in \N$ and $r>0$. 
	Suppose that \Cref{thm:BigTrees} holds for $r-1$.
	Given a $(k,r)$-pattern $\epsilon$ we define the expressions
	$\lambda_{r,\epsilon}$ and $\mu_{r,\epsilon}$ as follows.
	Let $(e,\tau)$ be a representative of $\epsilon$
	whose root is $v$. Then for all vertices $u\in V(e)$ such that
	$u\neq v$ it holds that $\tau(u)$ is a $\sim_k$ class of trees with
	radius at most $r$ and we can set  
	\[
	\lambda_{r,\epsilon}:=\prod_{\substack{u\in V(e)\ u\neq v} } \mathrm{Pr}\big[
	r-1, \tau(u)\big], \quad \text{ and } \quad
	\mu_{r,\epsilon}=\frac{\beta_{R(e)}}{\aut(\epsilon)} 
	  \lambda_{r,\epsilon}.
	\]
\end{definition}	
Clearly the definitions of $\lambda_{r,\epsilon}$ and
$\mu_{r,\epsilon}$ are independent of the chosen representative.
By hypothesis it holds
that $\mu_{r,\epsilon}$ is positive for all values of
$\{ \beta_R \}_{R\in \sigma}\in (0,\infty)^{|\sigma|}$
and it is an expression belonging to $M$.

\begin{lemma}\label{lem:poisedges}
	Let $k\in \N$, $r>0$ and $\overline{u}\in (\N)_*$.
	Let $\pi(\overline{x})\in FO[\sigma]$ be a consistent
	edge sentence such that 
	$\len(\overline{x})=\len(\overline{u})$.
	Let $\overline{v}\in (\N)_*$ be vertices contained
	in $\overline{u}$. For each $v\in \overline{v}$
	set $T_{n,v}:= \Tr\left(G_n, \overline{u};\,\,v;\,\,r\right)$.
	Given a pattern $\epsilon\in P(k,r)$ and $v\in \overline{v}$
	we define the random variable $X_{n,v,\epsilon}$ as
	the number of initial edges $e\in E(T_{n,v})$ such 
	that $(e,\tau^k_{(T_{n,v},v)})\in \epsilon$. Suppose that
	\Cref{thm:BigTrees} holds for $r-1$. Then the conditional 
	distributions of the variables $X_{n,v,\epsilon}$ given
	$\pi(\overline{u})$ converge to independent Poisson distributions
	whose respective mean values are given by the $\mu_{r,\epsilon}$. 
\end{lemma}
\begin{proof}
	To avoid excessively complex notation we prove only the case where
	$\overline{v}$ consists of a single vertex $v$. The general case is proven
	using the same arguments. Set $T_n:=T_{n,v}$ and 
	$X_{n,\epsilon}:=X_{n,v,\epsilon}$ for all $\epsilon\in P(k,r)$.
	By \Cref{thm:BrunSieve},
	in order to prove the result it is enough to show that 
	for any choice of natural numbers $\{b_\epsilon\}_{\epsilon\in P(k,r)}$
	it holds that
	\begin{equation} \label{eqn:binomexpedges}
	\Ln
	\mathrm{E}_{\pi(\overline{u})}
	\left[
	\prod_{\epsilon\in P(k,r)} \binom{X_{n,\epsilon}}{b_\epsilon}	
	\right]
	= \prod_{\epsilon\in P(k,r)} \frac{(\mu_{r,\epsilon})^
		{b_\epsilon}}{b_\epsilon!}.	
	\end{equation}
	Consider the numbers $\{b_\epsilon\}_{\epsilon\in P(k,r)}$
	fixed. 
	For each $n\in \N$ define
	\[
	\Omega_n:=\left\{
	\{E_\epsilon\}_{\epsilon\in P(k,r)} \quad 
	\Big | \quad \forall \epsilon\in P(k,r) \quad
	E_\epsilon\subset Copies(\epsilon,[n],(v,\rho)), 
	\quad |E_\epsilon|=b_\epsilon	
	\right\}.		
	\]
	Informally, elements of $\Omega_n$ represent choices of 
	$b_\epsilon$ possible initial edges of $T_n$ whose $k$-
	pattern is $\epsilon$ for all $(k,r)$-patterns $\epsilon$. 	
	Using \Cref{obs:binomialmean} we obtain
	\[
	\mathrm{E}_{\pi(\overline{u})}
	\left[
	\prod_{\epsilon\in P(k,r)} \binom{X_{n,\epsilon}}{b_\epsilon}	
	\right]
	= 
	\sum_{
		\{E_\epsilon\}_{\epsilon}
		\in \Omega_n}
	\mathrm{Pr}_{\pi(\overline{u})}\left(
	\bigwedge_{\substack{
			\epsilon\in P(k,r)\\
			(e, \tau)\in E_{\epsilon}
	}} \left(
	e\in E(T_n) \bigwedge_{\substack{
			u\in V(e)\\
			u\neq v}} Tr(T_n,v;u)\in \tau(u)		
	\right)
	\right). 		
	\]
	We say that a choice $\{E_\epsilon\}_{\epsilon}$
	$\in \Omega_n$ is \textbf{disjoint} if the edges
	$(e,\tau)\in\bigcup_{\epsilon\in P(k,r)} E_\epsilon$ satisfy
	that no vertex $w\in \overline{u}$ other than $v$ belongs to any
	of those edges and each vertex 
	$w\in [n]\setminus\{v\}$ belongs to at most one of those edges. 
	For each $n\in \N$ let $\Omega_n^\prime\subset \Omega_n$ 
	be the set	of disjoint elements in $\Omega_n$ and set
	$\Omega^\prime_\N= \cup_{n\in \N} \Omega^\prime_n$.
	If for some
	$\{E_\epsilon\}_\epsilon \in \Omega_n$ 
	we have that $e\in E(T_n)$ for all $(e,\tau)\in
	\bigcup_{\epsilon\in P(k,r)} E_\epsilon$ 
	then $\{E_\epsilon\}_\epsilon$ 
	is necessarily disjoint. 
	This is because $T_n$ is a tree and
	the only vertex in  $\overline{u}$ that belongs to $T_n$ 
	is $v$ by definition.  
	Thus, in the last sum it suffices to consider only 
	the disjoint $\{E_\epsilon\}_\epsilon$.
	Because of the symmetry of the random model the probabilities
	in that sum are the same for all disjoint choices of
	$\{E_\epsilon\}_\epsilon$.
	Hence, if we fix
	$\{E_\epsilon\}_{\epsilon}\in \Omega^\prime_\N$
	we obtain
	\begin{equation} \label{eqn:aux2}
	\mathrm{E}_{\pi(\overline{u})}
	\left[
	\prod_{\epsilon\in P(k,r)} \binom{X_{n,\epsilon}}{b_\epsilon}	
	\right]
	= 
	|\Omega_n^\prime| 
	\mathrm{Pr}_{\pi(\overline{u})}\left(
	\bigwedge_{\substack{
			\epsilon\in P(k,r)\\
			(e, \tau)\in E_{\epsilon}
	}} \left(
	e\in E(T_n) \bigwedge_{\substack{
			u\in V(e)\\
			u\neq v}} Tr(T_n,v;u)\in \tau(u)		
	\right)
	\right). 		
	\end{equation}
	
	Set $N:=\sum_{\epsilon\in P(k,r)} 
	(|\epsilon|-1)  b_\epsilon$.
	Counting vertices and automorphisms we get that
	\begin{equation} \label{eqn:aux3}
	|\Omega_n^\prime|= (n-\len(\overline{u}))_{N}
	\prod_{\epsilon\in P(k,r)}
	\frac{1}{b_\epsilon!}  
	\left( \frac{1}{\aut(\epsilon)} \right)^
	{b_\epsilon} .
	\end{equation}
	Let $\overline{w}\in (\N)_*$ be a list containing exactly
	the vertices $u\in V(e)$ for all $e\in 
	\bigcup_{\epsilon\in P(k,r)} E_\epsilon$. 
	Clearly, the event 
	\[ \bigwedge_{\substack{
			\epsilon\in P(k,r)\\
			(e, \tau)\in E_{\epsilon}
	}} e\in E(G_n)
	\]  can be described via an edge sentence
	whose variables are interpreted as vertices in $\overline{w}$.
	Let $\psi(\overline{x})$ be one of such edge sentences.
	This event is independent of $\pi(\overline{u})$ because edges are
	independent in $G_n$. Thus, a simple computation yields
	\[ 
	\mathrm{Pr}_{\pi(\overline{u})}
	\left(\bigwedge_{\substack{
			\epsilon\in P(k,r)\\
			(e, \tau)\in E_{\epsilon}
	}} e\in E(G_n)\right) = \prod_{\epsilon\in
	P(k,r)} \left(
	\frac{\beta_{R(\epsilon)}}{n^{ar(R(\epsilon)-1)}}
	\right)^{b_\epsilon}=
	\frac{1}{n^N}
	\prod_{\epsilon\in
	P(k,r)} \beta_{R(\epsilon)}^{b_\epsilon}.
	\]
	
	Because of \Cref{lem:far_away} a.a.s if $e\in E(G_n)$
	and $v\in V(e)$, then $e\in E(T_n)$. Thus,
	\begin{align} \label{eqn:aux4}
	&\mathrm{Pr}_{\pi(\overline{u})}\left(
	\bigwedge_{\substack{
			\epsilon\in P(k,r)\\
			(e, \tau)\in E_{\epsilon}
	}} \left(
	e\in E(T_n) \bigwedge_{\substack{
			u\in V(e)\\
			u\neq v}} Tr(T_n;\, u)
	\in \tau(u)		
	\right)
	\right)\sim \\ \nonumber
	&
	\left(\frac{1}{n^N}
	\prod_{\epsilon\in
	P(k,r)} \beta_{R(\epsilon)}^{b_\epsilon}\right)
	\mathrm{Pr}_{\pi(\overline{u})\wedge \psi(\overline{w})}\left(
	\bigwedge_{\substack{
			\epsilon\in P(k,r)\\
			(e, \tau)\in E_{\epsilon}
	}} 
	\bigwedge_{\substack{
		u\in V(e)\\
		u\neq v
	}}Tr(T_n;\, u)\in \tau(u)	
	\right).
	\end{align}
	The trees $Tr(T_n;u)$ in the last probability
	coincide with 
	$Tr(G_n,\overline{u}^\smallfrown\overline{w};\,u;\, r-1)$
	for all $u$.
	As a consequence, using the hypothesis that \Cref{thm:BigTrees}
	holds for $r-1$, we obtain
	\begin{align*}
	&	\mathrm{Pr}_{\pi(\overline{u})\wedge \psi(\overline{w})}
	\left(
	\bigwedge_{\substack{
			\epsilon\in P(k,r)\\
			(e, \tau)\in E_{\epsilon}
	}} 
	\bigwedge_{\substack{
			u\in V(e)\\
			u\neq v
	}}Tr(T_n;\, u)\in \tau(u)	
	\right) \sim
	 \prod_{\epsilon\in P(k,r)} (\lambda_{r,\epsilon})^{b_\epsilon}.
	\end{align*}
	Combining this this with \Cref{eqn:aux2,,eqn:aux3,,eqn:aux4} we obtain 
	\begin{align*} 
	\mathrm{E}_{\pi(\overline{u})}
	\left[
	\prod_{\epsilon\in P(k,r)} \binom{X_{n,\epsilon}}{b_\epsilon}	
	\right]
	\sim   \\
	\frac{(n-\len(\overline{u}))_{N}}
	{n^{N}}
	\prod_{\epsilon\in P(k,r)}
	\frac{1}{b_\epsilon!}  
	\left( \frac{\beta_{R(\epsilon)}  \lambda_{r,\epsilon}
	}{\aut(\epsilon)} \right)^
	{b_\epsilon} \sim \prod_{\epsilon\in P(k,r)}
	\frac{\left( \mu_{r,\epsilon} \right)^
		{b_\epsilon}}{b_\epsilon!}.
	\end{align*}
	This proves \Cref{eq:aux1} and the statement. 	
\end{proof}

Next lemma completes the proof of \Cref{thm:BigTrees}.

\begin{lemma} \label{lem:singletreeprob}
Let $r>0$. Suppose that \Cref{thm:BigTrees}
holds for $r-1$. Then it also holds for~$r$. 
\end{lemma}
\begin{proof}
Fix $k\in \N$. 
We start showing condition (1) of \Cref{thm:BigTrees}.
Fix $\mathbf{T}$ a $\sim_k$ class of trees with radius at most $r$. 
Fix a vertex $v\in \N$ as well. Set $T_n:=\Tr(G_n,v;\,v;\,r)$.
For each $\epsilon\in P(k,r)$ let
$X_{n,\epsilon}$ be the random 
variable that counts the number
of initial edges in $T_n$ whose pattern is 
$\epsilon$. Let $E^1_\mathbf{T}, E^2_\mathbf{T}, 
\{a_\epsilon\}_\epsilon$ be as in \Cref{obs:equivalenttrees}.
Then
\[
\Pr[r,\mathbf{T}]=
\Ln 
\mathrm{Pr}(T_n\in \mathbf{T})
=
\Ln
\mathrm{Pr}\left(
\left(
\bigwedge_{\epsilon\in E^1_\mathbf{T}}
X_{n,\epsilon}\geq k
\right) \wedge
\left(
\bigwedge_{\epsilon\in E^2_\mathbf{T}}
X_{n,\epsilon}= a_\epsilon
\right)
\right).
\]
Using the previous lemma we obtain that the last limit
equals the following expression:
\[ 
\left(
\prod_{\epsilon\in E^1_\mathbf{T}}
\mathrm{Poiss}_{\mu_{r,\epsilon}}(\geq k ) 
\right) 
\left(
\prod_{\epsilon\in E^2_\mathbf{T}}
\mathrm{Poiss}_{\mu_{r,\epsilon}}(a_\epsilon) 
\right).
\]
Using the definition of 
the $\mu_{r,\epsilon}$ we obtain that the last expression
belongs to $\Lambda$ as we wanted to prove. Furthermore,
as the $\mu_{r,\epsilon}$ are positive, this expression is also 
positive for all values of
$\{\beta_R\}_{R\in \sigma}\in (0,\infty)^{|\sigma|}$.
Now we proceed to prove condition (2). 
Let $\overline{u},\overline{v}, \{\mathbf{T}_v\}_{v\in \overline{v}}$
and $\pi(\overline{x})$ be as in the statement of (2). 
Using the previous lemma we obtain that the events
$Tr(G_n,\overline{u};\, v;\, r)\in \mathbf{T}_v$ for all $v\in \overline{v}$
are asymptotically independent and are also independent of $\pi(\overline{u})$. 
Then the desired result follows from condition (1). 
\end{proof}

\subsection{Almost all graphs are (k,r)-rich}

\begin{theorem}\label{thm:rich}
	Let $k,r\in \N$. Then a.a.s $G_n$ is $(k,r)$-rich.
\end{theorem}
\begin{proof}
	Let $\Sigma$ be the set of all $\sim_k$ classes of rooted trees
	with radii at most $r$. Let $m>k$. 
	For each $\mathbf{T}\in \Sigma$
	let $\overline{v}(\mathbf{T})\in (\N)_m$ be tuples
	satisfying 
	that all the $\overline{v}(\mathbf{T})$ are
	disjoint. Let $\overline{w}\in (\N)_*$ be a concatenation
	of all the $\overline{v}(\mathbf{T})$. 
	For each $\mathbf{T}\in 
	\Sigma$ define $X_{n,\mathbf{T}}$ as the number of
	vertices $v\in \overline{v}(\mathbf{T})$ such that 
	$Tr(G_n,\,\overline{w};\,v;\,r)\in \mathbf{T}$. Because of \Cref{thm:BigTrees}
	the $\sim_k$ types of the trees $Tr(G_n,\,\overline{w};\,v;\,r)$ 
	for all $v\in \overline{w}$ are asymptotically 
	independent and given any $v\in \overline{w}$ and
	$\mathbf{T}$ it holds that $\mathrm{Pr}(Tr(G_n,\,\overline{w};\,v;\,r)\in \mathbf{T})$
	tends to $\Pr[r,\mathbf{T}]$ as $n$ goes to infinity. 
	Hence,	the variables $X_{n,\mathbf{T}}$ converge in distribution to 
	independent binomial variables whose respective parameters are
	$m$ and $\Pr[r,\mathbf{T}]$. That is, given natural numbers
	$0\leq l_\mathbf{T} \leq m$ for all $\mathbf{T}\in \Sigma$, 
	\[
	\Ln
	\mathrm{Pr}\left(
	\bigwedge_{\mathbf{T}\in \Sigma} X_{n,\mathbf{T}}=l_\mathbf{T}
	\right)= 
	\prod_{\mathbf{T}\in \Sigma} \binom{m}{l_\mathbf{T}}
	\Pr[r,\mathbf{T}]^{l_\mathbf{T}}  (1-\Pr[r,\mathbf{T}] )^{m-l_\mathbf{T}}.
	\] 
	Fix $\delta>0$ such that $\delta< \Pr[r,\mathbf{T}]$ 
	for all $\mathbf{T}\in \Sigma$ and fix $\epsilon>0$ arbitrarily small.
	Because of the Law of large numbers, if $m$ is large enough
	\begin{equation}\label{eq:rich1}
	\Ln
	\mathrm{Pr}\left(
	\big| X_{n,\mathbf{T}}/m - \Pr[r,\mathbf{T}] \big| \geq \delta	
	\right) \leq \epsilon \quad \text{ for all $\mathbf{T}\in \Sigma$.}
	\end{equation}
	Also, for $m$ large enough we have 
	\begin{equation}\label{eq:rich2}
		\Pr[r,\mathbf{T}] > k/m + \delta \quad \text{for all $\mathbf{T}\in \Sigma$}.
	\end{equation}
	Suppose that $m$ is large enough for both \Cref{eq:rich1,eq:rich2}
	to hold. Then 
	\[
	\Ln
	\mathrm{Pr}\left(
	X_{n,\mathbf{T}} < k  \right) \leq \epsilon \quad \text{ for all $\mathbf{T}\in \Sigma$}
	\]
	We define $A_n$ as the event that
	for any $v\in \overline{w}$ we have
	$N(v;r)\cap Core(G_n;r)=\emptyset$ 
	(in particular this implies that $N(v;r)$ is a tree),
	and for any two $v_1,v_2\in \overline{w}$
	it is satisfied that	
	$d^{G_n}(v_1,v_2)>2r+1$.
	If $A_n$ holds then	for all	$v\in \overline{w}$ we have that
	$N(v;r)=Tr(G_n,\,\overline{w};\,v;\,r)$ and
	the $N(v;r)$ are disjoint trees. 
	Thus, if both $A_n$ holds and $X_{n,\mathbb{T}}\geq k$ for all $\mathbb{T}$ then 
	$G_n$ is $(k,r)$-rich. Because of \Cref{lem:disjointtrees} a.a.s $A_n$ holds, and we
	obtain
	\begin{align*}
	\Ln \mathrm{Pr}\left(
	\text{ $G_n$ is not
	$(k,r)$-rich }	
	\right) &\leq 
	\Ln \mathrm{Pr}\left(
	A_n \wedge \left(
	\bigvee X_{n,\mathbf{T}}<k
	\right)
	\right)\\ & =
	\Ln \mathrm{Pr}\left(
	\bigvee X_{n,\mathbf{T}}<k
	\right)	
	\leq \epsilon^{|\Sigma|}.
	\end{align*}
	As $\epsilon$ can be arbitrarily small given a suitable choice of $m$ we
	obtain that necessarily a.a.s $G_n$ is $(k,r)$-rich, as was to be proved. 
\end{proof}
		
\subsection{Probabilities of cycles}
	
\begin{definition}
	We define $\Gamma$ and $\Upsilon$ as the minimal families of expressions with arguments
	$\{\beta_R\}_{R\in \sigma}$ that satisfy the following conditions:
	(1) given natural numbers $a_R$ for each $R\in \sigma$, a positive number $b\in \N$
	and a $\lambda\in \Lambda$, the expression $\frac{\lambda}{b}  \prod_{R\in \sigma} \beta_R^{a_R}$
	belongs to $\Gamma$, 
	(2) given a $\gamma\in \Gamma$ and a $a\in \N$, the expressions $\mathrm{Poiss}_\gamma(a)$
	and $\mathrm{Poiss}_{\gamma}(\geq a)$ both belong to $\Upsilon$, and
	(3) if $\upsilon_1, \upsilon_2\in \Upsilon$ then 
	$\upsilon_1  \upsilon_2 \in \Upsilon$ as well. 
\end{definition}

\begin{definition} Let $k,r\in \N$ and $O\in C(k,r)$.
	Let $(H,\tau)$ be
	a representative of $O$.
	We define $\lambda_{r,O}$ and 
	$\gamma_{r,O}$ in the following way:
	\[ 
	\lambda_{r,O}:
	=\prod_{v\in V(H)}
	\mathrm{Pr} \big[ r, \tau(v) \big],
	\quad 
	\text{ and } \quad
	\gamma_{r,O}:=
	\frac{\prod_{R\in \sigma} \beta_R^{|E_R(H)|}}
	{\aut(H,\tau)} \lambda_{r,O}.
	\]	
	Clearly the definitions of $\lambda_{r,O}$ and $\gamma_{r,O}$
	are independent of the chosen representative and
	the expression $\gamma_{r,O}$ belongs to $\Gamma$.
\end{definition}

\begin{lemma}
	Let $k,r\in \N$. For any $O\in C(k,r)$ let $X_{n,O}$ be
	the random variable equal to the number of connected components
	$H$ of $Core(G_n;\, r)$ such that $H^\prime:=Center(H)$
	satisfies that $(H^\prime, \tau^k_{H})\in O$. Then the
	$X_{n,O}$ converge in distribution to independent 
	Poisson variables whose respective expected values are given by the 
	$\gamma_{r,O}$.
\end{lemma}
\begin{proof}
The proof is similar to the one of \Cref{lem:poisedges}. 
By \Cref{thm:BrunSieve}, to prove the result is enough to show
that for any natural numbers $\{b_O\}_{O\in C(k,r)}$ it holds 
\begin{equation} \label{eq:cycl_aux}
\Ln 
\mathrm{E}\left[
\prod_{O\in C(k,r)}
\binom{X_{n,O}}{b_O}
\right]= \prod_{O\in C(k,r)} 
\frac{(\gamma_{r,O})^{b_O}}{b_O!}.
\end{equation}
For each $n\in \N$ we define
\[
\Omega_n:=\left\{
\{F_O\}_{O\in C(k,r)} \quad \Big|
\quad \forall O\in C(k,r) \quad
F_O\subset Copies(O,[n]), \quad
|F_O|=b_O	
\right\}.
\]
Given a cycle $H$ such that $V(H)\subseteq [n]$ we say 
that $H\sqsubset G_n$ if $H=Center(H^\prime)$ for some connected 
component $H^\prime$ of $Core(G_n;\, r)$.
Using observation \Cref{obs:binomialmean} we obtain
\[
\mathrm{E}\left[
\prod_{O\in C(k,r)}
\binom{X_{n,O}}{b_O}
\right]=
\sum_{\{F_O\}_{O}\in \Omega_n}
\mathrm{Pr}\left(
\bigwedge_{
	\substack{
		O\in C(k,r)\\
		(H,\tau)\in F_O
}}
\left(
H\sqsubset G_n
\bigwedge_{v\in V(H)}
Tr(G_n,v;r)\in \tau(v)
\right)
\right).
\]
We call a choice $\{F_O\}_O\in \Omega_n$ \textbf{disjoint} if no 
vertex $v\in [n]$ belongs to two cycles $(H,\tau)\in \cup_O \, F_O$. 
Define $\Omega_n^\prime$ as the set of disjoint elements in $\Omega_n$
and set $\Omega_\N^\prime:=\cup_{n\in \N} \Omega^\prime_n$.
If for some $\{F_O\}_O\in \Omega_n$ it holds that $H \sqsubset G_n$
for all $(H,\tau)\in \cup_O F_O$ then necessarily $\{F_O\}_O$ is disjoint. 
Indeed, suppose the opposite. Then for some $(H_1,\tau_1), (H_2,\tau_2)\in \cup_O F_O$ 
it holds that $V(H_1)\cap V(H_2)\neq \emptyset$. Then both $H_1$ and $H_2$ belong
to the same connected component $H$ of $Core(G_n;\,r)$ and thus $H_1\cup H_2\subset
Center(H)$. As a consequence neither $H_1\sqsubset G_n$ or $H_2\sqsubset G_n$ hold. 
$(H_1,\tau_1),(H_2,\tau_2)\in \bigcup_{O\in C(k,r)} F_O$. Hence 
in the last sum it suffices to consider disjoint choices $\{F_O\}_O$. 
Because of the symmetry of the random model the probability
in that sum is the same for all disjoint choices of
$\{F_O\}_O$.
In consequence, if we fix
$\{F_O\}_{O}\in \Omega^\prime_\N$
we obtain
\begin{equation}\label{eq:cycl_aux1}
\mathrm{E}\left[
\prod_{O\in C(k,r)}
\binom{X_{n,O}}{b_O}
\right]=
|\Omega^\prime_n| 
\mathrm{Pr}\left(
\bigwedge_{
	\substack{
		O\in C(k,r)\\
		(H,\tau)\in F_O
}}
\left(
H\sqsubset G_n
\bigwedge_{v\in V(H)}
Tr(G_n,v;r)\in \tau(v)
\right)
\right).
\end{equation}

Set $N:=\sum_{O\in C(k,r)} |O|  b_O$. We have that
\begin{equation} \label{eq:cycl_aux2}
|\Omega_n^{\prime}|=
\frac{(n)_N}
{\prod_{O\in C(k,r)} b_O!  \aut(O)^{b_O}}.
\end{equation}
Let $\overline{v}\in (\N)_*$ be 
a list that contains exactly the vertices in $G\left(
\{F_O\}_{O\in C(k,r)}	\right)$. Then the event
\[
\bigwedge_{
	\substack{
		O\in C(k,r)\\
		(H,\tau)\in F_O\\
}}
H\subset G_n
\]
can be written as an edge sentence concerning the vertices in $\overline{v}$.
Let $\varphi(\overline{x})$ be one of such sentences. We have that
\[
\mathrm{Pr}\left(
\bigwedge_{
	\substack{
		O\in C(k,r)\\
		(H,\tau)\in F_O\\
}}
H\subset G_n
\right)= \prod_{O\in C(k,r)} \left( \frac{\prod_{R\in \sigma}
\beta_R^{|E_R(O)|}}{n^{|O|}} \right)^{b_O}=
\frac{1}{n^N}  
\prod_{O\in C(k,r)} \left( \prod_{R\in \sigma}
\beta_R^{|E_R(O)|} \right)^{b_O}.
\]
Because of \Cref{cor:simple} a.a.s if some cycle $H$ of diameter at most
$2r+1$ satisfies $H\subset G_n$ then $H\sqsubset G_n$. Hence,
\begin{align}
\nonumber
&\mathrm{Pr}\left(
\bigwedge_{
	\substack{
		O\in C(k,r)\\
		(H,\tau)\in F_O
}}
\left(
H\sqsubset G_n
\bigwedge_{v\in V(H)}
Tr(G_n,v;r)\in \tau(v)
\right)
\right) \sim \\ \label{eq:cycl_aux3}
&\frac{1}{n^N}  
\prod_{O\in C(k,r)} \left( \prod_{R\in \sigma}
\beta_R^{|E_R(O)|} \right)^{b_O}  
\mathrm{Pr}_{\varphi(\overline{v})}\left(
\bigwedge_{
	\substack{
		O\in C(k,r)\\
		(H,\tau)\in F_O
}}
\bigwedge_{v\in V(H)}
Tr(G_n,v;r)\in \tau(v)
\right).
\end{align}
As all the vertices $v\in \overline{v}$ belong to $Core(G_n;\,r)$, the
trees $Tr(G_n;\,v;\,r)$ in the last probability coincide with 
$Tr(G_n,\,\overline{v};\, v;\, r)$. By \Cref{thm:BigTrees}
we have that
\begin{equation*}
\mathrm{Pr}_{\varphi(\overline{v})}\left(
\bigwedge_{
	\substack{
		O\in C(k,r)\\
		(H,\tau)\in F_O
}}
\bigwedge_{v\in V(H)}
Tr(G_n,v;r)\in \tau(v)
\right) \sim 
\prod_{O\in C(k,r)} \left(
\lambda_{r,O}
\right)^{b_O}.
\end{equation*}
Combining this with \Cref{eq:cycl_aux1,eq:cycl_aux2,eq:cycl_aux3} we obtain
\begin{equation*}
\mathrm{E}\left[
\prod_{O\in C(k,r)}
\binom{X_{n,O}}{b_O}
\right]\sim
\frac{(n)_N}{n^N}
\prod_{O\in C(k,r)}
\frac{1}{b_O!}
\left(
\frac{\lambda_{r,O} \prod_{R\in\sigma} \beta_R^{|E_R(O)|}}{\aut(O)}
\right) \sim \prod_{O\in C(k,r)} \frac{(\gamma_{r,O})^{b_O}}{b_O!}.
\end{equation*}
This proves \Cref{eq:cycl_aux} and the statement. 
\end{proof}

\begin{theorem}\label{thm:agreeabilityprobabilities}
	Let $k,r\in \N$ and let $\mathbf{O}$ be a simple $(k,r)$-agreeability class
	of hypergraphs. Then 
	$\Ln \PR{G_n\in \mathbf{O}}$ exists and is an expression
	in $\Upsilon$. 
\end{theorem}
\begin{proof}
For each $O\in C(k,r)$ let $X_{n,O}$ be as in the previous lemma. 
Let $U_\mathbf{O}^1, U_\mathbf{O}^2$ and $\{a_O\}_{O\in U^2_\mathbf{O}}$ 
be as in \Cref{obs:agreeablecores}. Let $A_n$ be the event that $G_n$ is $r$-simple. 
Then
\begin{align*}
 \Ln
 \PR{G_n \in \mathbf{O}}=
 \Ln
 \mathrm{Pr} \left(
 A_n \wedge
 \left(
 \bigwedge_{O\in U_\mathbf{O}^1}
 X_{n,O}\geq k
 \right)
 \wedge
 \left(
 \bigwedge_{O\in  U^2_\mathbf{O}}
 X_{n,O}=a_O.
 \right)
 \right).
\end{align*}
Because of \Cref{cor:simple}, a.a.s $A_n$ holds. Thus, using the last
lemma the previous limit equals the following expression
\begin{align*}
\left(
\prod_{O\in C_1}
\mathrm{Poiss}_{\gamma_{r,O}}(\geq k)
\right)
\left(
\prod_{O\in C_2}
\mathrm{Poiss}_{\gamma_{r,O}}(a_O)
\right)
.
\end{align*}
As all the $\gamma_{r,O}$ belong to $\Gamma$,
this last expression belongs to $\Upsilon$ and the theorem is proven. 
\end{proof}

\section{Proof of the main theorem} \label{sect:main}

\begin{theorem}
	Let $\phi\in FO[\sigma]$. Then the function 
	$F_\phi: [O,\infty)^{|\sigma|}\rightarrow [0,1]$
	given by
	\[
	\{\beta_R\}_{R\in \sigma} \mapsto
	\Ln \mathrm{Pr} \left(
	G_n\left(\{\beta_R\}_{R}\right) \models \phi
	\right)
	\]
	is well defined and it is given by a finite sum of expressions
	in $\Upsilon$.
\end{theorem}
\begin{proof}
	Let $k$ be the quantifier rank of $\phi$ and
	let $r=3^k$. Let 
	$G_n:=G_n\left(\{\beta_R\}_{R\in \sigma}\right)$ and 
	let $\Sigma$ be the set of $(k,r)$-agreeability classes of 
	$r$-simple hypergraphs. Because of \Cref{cor:simple} a.a.s 
	$G_n$ is $r$-simple. Thus
	\begin{equation} \label{eq:aux1}
	\Ln \mathrm{Pr} \left(
	G_n \models \phi
	\right)=
	\Ln
	\sum_{\mathbf{O}\in \Sigma} \mathrm{Pr}\left(
	G_n\in \mathbf{O}
	\right)   
	\mathrm{Pr}\left(
	G_n\models \phi \,
	\Big| \,
	G_n\in \mathbf{O}
	\right).
	\end{equation}
	Because the set $\Sigma$ is finite, we can exchange the 
	summation and the limit. By \Cref{thm:rich} a.a.s 
	$G_n$ is $(k,r)$-rich. This together with \Cref{thm:Duplicatorwins}
	implies that for any $\mathbf{O}\in \Sigma$
	\[
	\Ln \mathrm{Pr}\left(
	G_n\models \phi \,
	\Big| \,
	G_n\in \mathbf{O}
	\right) = 
	\text{ $0$ or $1$ }.
	\]
	Let $\Sigma^\prime\subset \Sigma$ be the set of classes 
	$\mathbf{O}$ for which last limit equals $1$. Then
	\[
	\Ln \mathrm{Pr} \left(
	G_n \models \phi
	\right)=
	\sum_{\mathbf{O}\in \Sigma^\prime}
	\Ln \mathrm{Pr}\left(
	G_n \in \mathbf{O}
	\right).		
	\]
	Because of \Cref{thm:agreeabilityprobabilities} we know that
	each of the limits inside the last sum exists and is given by
	an expression that belongs to $\Upsilon$. As a consequence the
	theorem follows. 	
\end{proof}

\section{Application to random SAT}\label{sec:SAT}

We define a binomial model of random CNF formulas, 
in analogy with the one in \cite{chvatal1992mick},
but the generality
in \Cref{thm:main} allows for many variants. \par

\begin{definition}\label{def:CNF}
Given a variable $x$, both expressions
$x$ and $\neg x$ are called \textbf{literals}. A \textbf{clause}
is a set of literals. A clause $C$ is called \textbf{non-tautological}
if no variable $x$ satisfies that both $x$ and $\neg x$
belong to $C$. An \textbf{assignment} over a set of variables $X$ is a 
map $f$ that assigns $0$ or $1$ to each variable of $X$. A clause $C$
is \textbf{satisfied} by an assignment $f$ if either there is some variable $x$
such that $x\in C$ and $f(x)=1$ or there is some variable $x$ such that
$\neg x\in C$ and $f(x)=0$. 
Given $l\in\N$
a $l$-\textbf{CNF formula} is a set of  non-tautological clauses 
that contain exactly $l$ literals. 
We say that a formula $F$
on the variables $x_1,\dots, x_n$ is \textbf{satisfiable} if there is an
assignment $f:\{x_1,\dots, x_n\}\rightarrow \{0,1\}$ that satisfies all clauses
in $F$. 
\end{definition}

Given $n, l \in \N$ and a real number $0\leq p \leq 1$ we define
the random model $F(l,n,p)$ as the discrete probability space that
assigns to each $l$-CNF formula $F$ on
the variables $\{x_i\}_{i\in [n]}$
 the probability
\[
\PR{F}= p^{|F|}  (1-p)^{2^l\binom{n}{l}-|F|},
\] 
where $|F|$ is the number of clauses in $F$.
Equivalently, a random formula in $F(l,n,p)$ is obtained
by choosing each  of the $2^l\binom{n}{l}$ non-tautological
clauses of size $l$ on the variables $\{x_i\}_{i}$
with probability $p$ independently. When
$p$ is a function of $n$ satisfying
$p(n)\sim \beta/n^{l-1}$ we denote by $F^l_n(\beta)$ a random sample
of $F(l,n,p(n))$.
\par	
We consider $l$-CNF formulas, as defined above, as relational 
structures with a language $\sigma$ consisting of $l+1$ relation symbols
$R_0,\dots, R_l$ of arity $l$. We do that in such a way that the expression
$R_j(x_{i_1},\dots,x_{i_l})$ means that our formula contains the clause
consisting of $\neg x_{i_1}, \dots, \neg x_{i_j}$ and $x_{i_{j+1}},\dots
x_{i_l}$. The relations $R_1,\dots, R_l$ satisfy the 
following axioms: (1) given $0\leq j \leq l$ and 
variables $y_1,\dots, y_l$ the fact that
$R_j(y_1,\dots, y_l)$ holds
is invariant under any permutation of the variables 
$y_1,\dots,y_j$ or $y_{j+1},\dots,y_l$, and (2) for 
any $0\leq j \leq l$ and 
any variables $y_1,\dots, y_l$ it holds that
$R_j(y_1,\dots, y_l)$ only if all the $y_i$ are different. 
Call $\mathcal{C}$ to the family of $\sigma$-structures satisfying the last two axioms.
The language $\sigma$ and the family $\mathcal{C}$ satisfy the conditions in 
\Cref{sect:structures}. The random model $F_l(n,p)$ coincides with the model
$G(n,\{p_R\}_{R})$ of random $\mathcal{C}$-hypergraphs described in 
\Cref{sect:random} when all the $p_R$ are equal. As a particular
case of \Cref{thm:main} we obtain the following result. 
\begin{theorem} \label{thm:mainsat}
	Let $l>1$ be a natural number.
	Then for each sentence $\Phi\in FO[\sigma]$ 
	it is satisfied
	that the map $f_\Phi: (0,\infty) \rightarrow \R$ given by
	\[
	\beta \mapsto \Ln \PR{ F^l_n(\beta)\models \Phi}
	\]
	is well defined and analytic. 
\end{theorem}
The following is a well known result regarding random CNF formulas. 
\begin{theorem} 
Let $l\geq 2$ be a natural number, and let $c\in (0,\infty)$ 
be an arbitrary real number. 
Let $m:\N\rightarrow \N$ be such that
$m(n)\sim c n$. For each $n$ let $C_{n,1},\dots, C_{n,m(n)}$
be clauses chosen uniformly at random independently among the 
$2^l \binom{n}{l}$ non-tautological clauses of size $l$ over the
variables $x_1,\dots, x_n$. For each $n$, let $UNSAT_n$
denote the event that there is no assignment of the variables
$x_1,\dots,x_n$ that satisfies all clauses $C_{n,1},\dots, C_{n,m(n)}$. 
Then there are two real constants $0<c_1<c_2$, such that 
a.a.s $UNSAT_n$ does not hold if $c< c_1$, and a.a.s $UNSAT_n$ holds
if $c> c_2$. 
\end{theorem}
The existence of $c_1$ is proven in \cite[Theorem 1]{chvatal1992mick}.
The fact that $c_2$ exists follows from a direct application of
the first moment method and is also shown for instance in \cite{chvatal1992mick,franco1983probabilistic,chvatal1988many}.
We want to show that an analogous ``phase transition" also happens in
$F(l,n,p)$ when $p\sim \beta/n^{l-1}$. We start by showing the following
\begin{corollary}
	Let $l\geq 2$ be a natural number. Let $c\in (0,\infty)$ be an arbitrary
	real number and let $m:\N\rightarrow \N$ satisfy $m(n)\sim c n$. 
	For each $n\in \N$ let $F_{n,m(n)}$ be a random formula chosen uniformly at
	random among all sets of $m(n)$ non-tautological clauses of size $l$ over
	the variables $x_1,\dots, x_n$. Then there
	are two real positive constants $0<c_1<c_2$ such that
	a.a.s $F_{n,m(n)}$ is satisfiable if $c< c_1$, and
	a.a.s $F_{n,m(n)}$ is unsatisfiable if $c> c_2$.
\end{corollary}
\begin{proof}
For each $n\in \N$ let $C_{n,1},\dots, C_{n,m(n)}$ and $UNSAT_n$ be as in the previous theorem.
One can consider $F_{n,m(n)}$ to be the result of selecting clauses $C_{n,1},\dots, C_{n,m(n)}$
uniformly at random independently among all possible clauses, given the fact that
no two clauses $C_{n,i}, C_{n,j}$ are equal. Hence,
\[
\mathrm{Pr} \left( F_{n,m(n)} \text{ is unsatisfiable } \right) 
=
\mathrm{Pr} \left(UNSAT_n \, \big| \, \text{ all the $C_{n,i}$  are different }\,\right).
\] 
An application of the first
moment method yields that for $l\geq 3$ a.a.s the number of unordered pairs $\{i,j\}$ such that
$C_{n,i}=C_{n,j}$ is equal to zero. In the case of $l=2$, an application of \Cref{thm:BrunSieve} proves that the number of such pairs $\{i,j\}$ converges in distribution to 
a Poisson variable. In either case all the
$C_{n,i}$  are different with positive asymptotic probability. 
Thus the constants $c_1$ and $c_2$ from the previous theorem satisfy
our statement. 
\end{proof}

Let $F_{n,m(n)}$ be as in last result. Note that because of the symmetry in 
the random model $F(l,n,p(n))$ one can consider $F_{n,m(n)}$ to be a random
sample of the space $F(l,n,p(n))$ given that the number of clauses
is $m(n)$. Using this observation we can prove the following. 
  
\begin{theorem} \label{thm:phasetransition}
	Let $l> 1$. Then there are real positive values 
	$\beta_1 < \beta_2$ such that a.a.s $F^l_n(\beta)$ is satisfiable
	for $0<\beta<\beta_1$ and a.a.s $F^l_n(\beta)$ is unsatisfiable
	and for $\beta>\beta_2$.
\end{theorem}
\begin{proof}
	For each $n\in \N$ let $X_{n}(\beta)$ be the random variable
	equal to the number of clauses in $F^l_n(\beta)$. We have that
	$\mathrm{E}[X_n(\beta)]\sim \frac{\beta  2^l}{l!}n$. Let $c_1,
	c_2$ be as in last corollary. Define $\beta_1:= \frac{c_1  l!}{2^l}$
	and $\beta_2:= \frac{c_2   l!}{2^l}$. Fix $\beta\in \R$ satisfying
	$0<\beta<\beta_1$. Let $\epsilon>0$ be a real number such that
	$ \frac{\beta  2^l}{l!} +\epsilon< c_1$. For each $n\in \N$
	set $\delta_1(n):=\floor*{ \left(\frac{\beta  2^l}{l!}-\epsilon\right)n}$
	and $\delta_2(n):=\floor*{ \left(\frac{\beta  2^l}{l!}+\epsilon\right)n}$.

	Denote by $dp_n$ the probability density function of the variable $X_n(\beta)$.
	That is $dp_n(m)=\mathrm{Pr}(X_n(\beta)=m)$. Then, because of the previous equation,
	\[
	\mathrm{Pr}\left( F^l_n(\beta) \text{ is unsatisfiable }\right)\sim
	\int_{\delta_1(n)}^{\delta_2(n)}
 \mathrm{Pr} \left(
	F^l_n(\beta) \text{ is unsatisfiable } \Big|
	X_n(\beta)=m	
	\right)  dp_n(m).
	\]
	Note that the property of being unsatisfiable is monotonous.
	As a consequence,
	\begin{align*}
	&\int_{\delta_1(n)}^{\delta_2(n)}
	\mathrm{Pr} \left(
	F^l_n(\beta) \text{ is unsatisfiable } \Big|
	X_n(\beta)=m	
	\right)  dp_n(m)\leq \\ & 
	\mathrm{Pr} \left(
	F^l_n(\beta) \text{ is unsatisfiable } \Big|
	X_n(\beta)=\delta_2(n) \right)   
	\mathrm{Pr}\left( \delta_1(n) \leq X_n(\beta)  
	\leq \delta_2(n)  \right).   
	\end{align*}
	Because of the Law of large numbers,
	\begin{equation*}
	\Ln \mathrm{Pr}\left( \delta_1(n) \leq X_n(\beta)  
	\leq \delta_2(n) \right) = 1.
	\end{equation*}
	As $\delta_2(n)< c_2n$, because of the previous
	corollary
	\[
	\Ln \mathrm{Pr} \left(
	F^l_n(\beta) \text{ is unsatisfiable } \Big|
	X_n(\beta)=\delta_2(n) \right)= 0.
	\]  
	Combining the previous equations we obtain that for any 
	$\beta < \beta_1$ it holds that
	$F^l_n(\beta)$ a.a.s is satisfiable, as it was to be proven.
	Showing that for any $\beta > \beta_2$, 
	a.a.s $F^l_n(\beta)$ is unsatisfiable is analogous.
\end{proof}

A direct consequence of the last theorem, due to A. Atserias (personal communication, July, 2019), is the following

\begin{theorem} \label{thm:satapplication}
	Let $l>1$ be a natural number. Let $\Phi\in FO[\sigma]$ be a first order
	sentence that implies unsatisfiability.  Then for all $\beta>0$
	a.a.s $F^l_n(\beta)$ does not satisfy $\Phi$.
\end{theorem}
\begin{proof}
	Let $\beta_1$ and $\beta_2$ be as in \Cref{thm:phasetransition}. 
	As $\Phi$ implies unsatisfiability
	$\PR{F^l_n(\beta)\models \Phi  }\leq  
	\PR{F^l_n(\beta) \text{ is unsatisfiable }  }$. Thus, by
	\Cref{thm:phasetransition},
	we get that for all $\beta\in (0,\beta_1]$
	\[
	\Ln \PR{F^l_n(\beta)\models \Phi  }=0.
	\]
	By \Cref{thm:mainsat}, last limit varies analytically with $\beta$. 
	It vanishes in the proper interval $(0,\beta_1]$ then by the Principle of analytic continuation it has to vanish
	in the whole $(0,\infty)$, and the result 
	holds. 
\end{proof}

\section*{Acknowledgments}
This work was supported by the European Research Council (ERC) under the European Union  
Horizon 2020 research and innovation programme (grant agreement ERC-2014-CoG 648276 AUTAR).\par
I would like to thank both my supervisor Marc Noy and Albert Atserias for suggesting the topic of this
research. I am grateful to M. Noy for introducing me to the topic of zero-one laws
and for helpful discussions on the subject. I am also thankful to A. Atserias for his insight on
random SAT problems and for suggesting the proof of \Cref{thm:mainsat}. The feedback given by both of them has been very helpful. 

\bibliographystyle{abbrvnat}
\bibliography{biblio}

\begin{thebibliography}{14}
\providecommand{\natexlab}[1]{#1}
\providecommand{\url}[1]{\texttt{#1}}
\expandafter\ifx\csname urlstyle\endcsname\relax
  \providecommand{\doi}[1]{doi: #1}\else
  \providecommand{\doi}{doi: \begingroup \urlstyle{rm}\Url}\fi

\bibitem[Atserias(2005)]{atserias2005definability}
A.~Atserias.
\newblock Definability on a random 3-cnf formula.
\newblock In \emph{20th Annual IEEE Symposium on Logic in Computer Science
  (LICS'05)}, pages 458--466. IEEE, 2005.

\bibitem[Bollob{\'a}s and B{\'e}la(2001)]{bollobas2001random}
B.~Bollob{\'a}s and B.~B{\'e}la.
\newblock \emph{Random graphs}.
\newblock Number~73. Cambridge university press, 2001.

\bibitem[Chv{\'a}tal and Reed(1992)]{chvatal1992mick}
V.~Chv{\'a}tal and B.~Reed.
\newblock Mick gets some (the odds are on his side)(satisfiability).
\newblock In \emph{Proceedings., 33rd Annual Symposium on Foundations of
  Computer Science}, pages 620--627. IEEE, 1992.

\bibitem[Chv{\'a}tal and Szemer{\'e}di(1988)]{chvatal1988many}
V.~Chv{\'a}tal and E.~Szemer{\'e}di.
\newblock Many hard examples for resolution.
\newblock \emph{Journal of the ACM (JACM)}, 35\penalty0 (4):\penalty0 759--768,
  1988.

\bibitem[Ebbinghaus and Flum(2005)]{finitemodeltheory1}
H.-D. Ebbinghaus and J.~Flum.
\newblock \emph{Finite model theory}.
\newblock Springer Science \& Business Media, 2005.

\bibitem[Ehrenfeucht(1961)]{ehrenfeucht1961application}
A.~Ehrenfeucht.
\newblock An application of games to the completeness problem for formalized
  theories.
\newblock \emph{Fund. Math}, 49\penalty0 (129-141):\penalty0 13, 1961.

\bibitem[Fagin(1976)]{fagin1976probabilities}
R.~Fagin.
\newblock Probabilities on finite models 1.
\newblock \emph{The Journal of Symbolic Logic}, 41\penalty0 (1):\penalty0
  50--58, 1976.

\bibitem[Franco and Paull(1983)]{franco1983probabilistic}
J.~Franco and M.~Paull.
\newblock Probabilistic analysis of the davis putnam procedure for solving the
  satisfiability problem.
\newblock \emph{Discrete Applied Mathematics}, 5\penalty0 (1):\penalty0 77--87,
  1983.

\bibitem[Glebskii et~al.(1969)Glebskii, Kogan, Liogon'kiI, and
  Talanov]{glebskii1969range}
Y.~V. Glebskii, D.~I. Kogan, M.~Liogon'kiI, and V.~Talanov.
\newblock Range and degree of realizability of formulas in the restricted
  predicate calculus.
\newblock \emph{Cybernetics and Systems Analysis}, 5\penalty0 (2):\penalty0
  142--154, 1969.

\bibitem[Lynch(1992)]{lynch1992probabilities}
J.~F. Lynch.
\newblock Probabilities of sentences about very sparse random graphs.
\newblock \emph{Random Structures \& Algorithms}, 3\penalty0 (1):\penalty0
  33--53, 1992.

\bibitem[Saldanha and Telles(2016)]{salvadorbrasil}
N.~C. Saldanha and M.~Telles.
\newblock Spaces of completions of elementary theories and convergence laws for
  random hypergraphs.
\newblock \emph{arXiv preprint arXiv:1602.06537}, 2016.

\bibitem[Shelah and Spencer(1988)]{shelah1988zero}
S.~Shelah and J.~Spencer.
\newblock Zero-one laws for sparse random graphs.
\newblock \emph{Journal of the American Mathematical Society}, 1\penalty0
  (1):\penalty0 97--115, 1988.

\bibitem[Shelah and Spencer(1994)]{shelah1994can}
S.~Shelah and J.~Spencer.
\newblock Can you feel the double jump?
\newblock \emph{Random Structures \& Algorithms}, 5\penalty0 (1):\penalty0
  191--204, 1994.

\bibitem[Spencer(2013)]{spencer2013strange}
J.~Spencer.
\newblock \emph{The strange logic of random graphs}, volume~22.
\newblock Springer Science \& Business Media, 2013.

\end{thebibliography}
	
\end{document}